\documentclass[11pt]{article}

\usepackage{graphicx}
\usepackage{amsmath,amsthm,amssymb,color,bbm}
\usepackage[noadjust,sort]{cite}
\usepackage[T1]{fontenc}
\usepackage{xcolor}

\usepackage{tikz}
\usetikzlibrary{positioning}
\usepackage[font=small,labelfont=bf]{caption}

\usepackage[normalem]{ulem}

\makeatletter
\g@addto@macro\normalsize{%
	\setlength\abovedisplayskip{8pt plus 3pt minus 3pt}
	\setlength\belowdisplayskip{8pt plus 3pt minus 3pt}
	\setlength\abovedisplayshortskip{6pt plus 3pt minus 2pt}
	\setlength\belowdisplayshortskip{6pt plus 3pt minus 2pt}
}
\makeatother

\date{\today}

\normalsize

\numberwithin{equation}{section}

\def\({\bigl(}
\def\){\bigr)}

\newtheorem{thm}{Theorem}[section]
\newtheorem{cor}[thm]{Corollary}
\newtheorem{lemma}[thm]{Lemma}
\newtheorem{conj}[thm]{Conjecture}

\theoremstyle{definition} 
\newtheorem{remark}[thm]{Remark}

\def\dfrac#1#2{\lower0.15ex\hbox{\large$\textstyle\frac{#1}{#2}$}}
\def\({\bigl(}
\def\){\bigr)}
\def\st{\,:\,}
\def\mid{\mathrel{|}}

\let\eps=\varepsilon

\def\calS{\mathcal{S}}
\def\calM{\mathcal{M}}

\def\calL{\mathcal{L}}

\def\calP{\mathcal{P}}

\def\1{\mathbbm{1}}

\def\X{\boldsymbol{X}}
\def\Y{\boldsymbol{Y}}
\def\Z{\boldsymbol{Z}}
\def\x{\boldsymbol{x}}

\def\n{\boldsymbol{n}}

\def\calF{\mathcal{F}}
\def\calG{\mathcal{G}}

\def\calC{\mathcal{C}}

\def\calS{\mathcal{S}}

\def\xvec{\boldsymbol{x}}
\def\yvec{\boldsymbol{y}}
\def\zvec{\boldsymbol{z}}
\def\uvec{\boldsymbol{u}}

\def\avec{\boldsymbol{a}}

\def\alphavec{\boldsymbol{\alpha}}

\def\one{\boldsymbol{1}}

\def\E{\operatorname{\mathbb{E}}}

\def\Pr{\mathbb{P}}

\def\Reals{{\mathbb{R}}}
\def\pReals{\Reals_+}

\def\Naturals{{\mathbb{N}}}

\def\esssup{\operatorname{ess\,sup}\limits}

\def\nb{k} 

\def \M{\mathcal{G}} 
\def \SBM{\M(\n,P)} 

\def \G{\boldsymbol{G}} 

\begin{document}

	\title{On the chromatic number of graphons}
	\author{
		Mikhail Isaev\thanks{Supported by Australian Research Council Discovery Project DP190100977
			and by Australian Research Council Discovery Early Career Researcher Award DE200101045}\\
		\small School of Mathematics\\[-0.8ex]
		\small Monash University\\[-0.8ex]
		\small 3800  Clayton,   Australia\\[-0.8ex]
		\small\tt  mikhail.isaev@monash.edu    
		\and
		Mihyun~Kang\thanks{Research supported in part by FWF I3747. Part of this work was done while the author was visiting the Simons Institute for the Theory of Computing.}\\
		\small Institute of Discrete Mathematics\\[-0.8ex]
		\small Graz University of Technology\\[-0.8ex]
		\small  8010  Graz, Austria\\[-0.8ex]
		\small\tt  kang@math.tugraz.at
	}
	\date{\today}


	\maketitle

	\begin{abstract}
		We extend  Bollobas' classical result on  the chromatic number of a binomial random graph  to  the exchangeable random graph model $\calG(n,W)$ defined by a graphon $W:[0,1]^2 \rightarrow [0,1]$, which is a symmetric measurable function.
		In the case when $W$ can be approximated by  block graphons in $\calL^{\infty}$-norm, we show that asymptotically optimal value  of the number of colours required for $\calG(n,W)$  is  determined by colouring strategies that use a finite number of different types of colour classes.   Furthermore, if $W$ is a block graphon with $k\times k$ blocks then $k$ types of colour classes are sufficient. We also show that if $W$ is   block-increasing or  block-Lipschitz  then such colouring strategies that use  $k$ types  determine the chromatic number up to a multiplicative error of order $O(k^{-1})$. 
%
%
	\end{abstract}

	\section{Introduction}\label{S:intro}

	
	The chromatic number of a graph is a fundamental parameter of a graph. Given a graph $G$, its chromatic number $\chi(G)$ is the smallest number of colours needed for the assignment of colours to the vertices of $G$ so that no two adjacent vertices have the same colour. 
	
	In their seminal paper \cite{ER60}, Erd\H{o}s and R\'enyi raised research questions/problems leading to the theory of random graphs. One of their problems that have attracted constant attention concerns the chromatic number of a random graph. The asymptotic behaviour and concentration of the chromatic number  of various random graph models was thoroughly investigated by many researchers: see, for example,~\cite{AchlioptasMoore,AchlioptasNaor2005, AlonKrivelevich1997, Bollobas1988, Bollobas2004, CPA2008, CooperFriezeReedRiordan, FriezeKrivelevichSmyth, FriezeLuczak, GaoOhapkin,  Heckel2018, Heckel2020, IsaevKang-SBM, KemkesGimenezWormald, KrivelevichSudakov1998, Luczak1991, Luczak1991b, MPSM2020, McDiarmid1990, McDiarmid_geometric, McDiarmidMueller_geometric, Penrose_geometric, Panagiotou2009, scott2008, ShamirSpencer1987, SW2022} and references therein.

	One of the most important results along this line of research is Bollobas' result~\cite{Bollobas1988} on the typical value of the chromatic number in  the binomial random graph model $\calG(n,p)$. Let   $p\in (0,1)$ be fixed and set $b=\frac{1}{1-p}$. Then, with high probability ({\em whp} for short, meaning with probability tending to one as $n\to \infty$),
	\begin{equation}\label{Bollobas-original}
		\chi(\G ) = (1+o(1)) \frac{n}{2 \log_{b} n}, \qquad \text{where $\G \sim \calG(n,p)$.}
	\end{equation}
	This result was strengthened and extended, in particular, by McDiarmid~\cite{McDiarmid1990}, by {\L}uczak~\cite{Luczak1991}, by Scott~\cite{scott2008}, by Panagiotou and Steger~\cite{Panagiotou2009}, and by Heckel~\cite{Heckel2018}. As for the concentration of the chromatic number, starting from the classical result by Shamir and Spencer~\cite{ShamirSpencer1987}, there have been breakthough results for various ranges of $p=p(n)$ by  {\L}uczak~\cite{Luczak1991b}, by Alon and Krivelevich~\cite{AlonKrivelevich1997}, and by Achlioptas and Naor~\cite{AchlioptasNaor2005},  to mention a few.

	Some of these results are extended to  random regular graphs; see, for example, the papers by Achlioptas and  Moore~\cite{AchlioptasMoore}, by Cooper, Frieze, Reed and Riordan~\cite{CooperFriezeReedRiordan}, by Frieze and {\L}uczak~\cite{FriezeLuczak}, and by Kemkes, P\'erez-Gim\'en\'ez and Wormald~\cite{KemkesGimenezWormald}. A uniform random $d$-regular graph model $\calG(n,d)$ is a typical example of {\em homogeneous} random graph models that is closely related to  $\calG(n,p)$. The intense study on $\calG(n,p)$  and $\calG(n,d)$ revealed that these two models share various properties, especially in the regime $d\gg \log n$, including the asymptotic behaviour of the chromatic number. 
%
%
	
	The chromatic number remains to be a central topic also in {\em inhomogeneous} random graph models, including random graphs with specified degrees, random geometric graphs, and the stochastic block model.  In contrast to {\em homogeneous} random graph models such as $\calG(n,p)$ and $\calG(n,d)$, in which we can colour almost all vertices asymptotically optimally by a {\em"greedy" colouring} that  assigns distinct colours to independent sets of  asymptotically same asymptotic size $(2+o(1)) \log_{b} n$,  the "greedy" colouring in inhomogeneous models typically produces sets of different sizes and involves vertices of several different types.  The resulting number of colours is often  bigger than in the \emph{"balanced" strategy}, where we enforce all independent sets to be similar.   
	Furthermore, both strategies might be far from optimal colourings for which   inhomogeneities of the model play much more  intricate role.

	The chromatic number of random graphs with given degrees is studied by Frieze,  Krivelevich and Smyth~\cite{FriezeKrivelevichSmyth} and very recently by Gao and Ohapkin~\cite{GaoOhapkin}. In case of random geometric graphs, the study on its chromatic number was initiated by McDiarmid~\cite{McDiarmid_geometric} and by Penrose~\cite{Penrose_geometric}, whose results were sharpened and extended by  McDiarmid and Müller~\cite{McDiarmidMueller_geometric}. Another classic inhomogeneous random graph model is the stochastic block model, a random graph with planted clusters or with community structure. In their recent work, Martinsson, Panagiotou, Su, and Truji\'c~\cite{MPSM2020} determined a typical value of the chromatic number in  the \emph{stochastic block model}. Isaev and Kang~\cite{IsaevKang-SBM} extended their work to allow the number of blocks to grow and the edge probabilities tend to zero, and they also determined the chromatic number in the Chung-Lu model.

	Turning our attention to the limit objects of sequences of dense graphs or to analytic objects that approximate very large networks, we consider graphons that are  generalisations of graphs. 
	The most important applications of graphons  include the limit theory of dense graphs by Lov\'{a}sz et al~\cite{BorgsChayesLovasz10,BCLSV06,LovaszSzegedy06,LovaszSzegedy2007,Lovasz2012}), 
	large deviation principles for random graphs by Chatterjee and Varadhan~\cite{ChatterjeeVaradhan11},  as well as property testing by Lov\'{a}sz and Szegedy~\cite{LovaszSzegedy10ijm}. Graphons are closedly related to exchangeable arrays and infinite exchangeable graphs studied by Hoover~\cite{Hoover79} and by Aldous ~\cite{Aldous81} as well as by Diaconis and Janson~\cite{DiaconisJanson08}.

	Formally, 	a \emph{graphon} is a symmetric measurable function $W: [0,1]^2 \rightarrow [0,1]$. 
	Given  $n \in \Naturals$ and a graphon $W$,  the exchangeable random graph model 
	$ \M(n,  W)$  is defined as follows.    
	First,  we  sample $X_1,\ldots,X_n\in [0,1]$    independently uniformly at random.
	A random graph from $\M(n,W)$
	has vertex set $[n]:= \{1,\ldots, n\}$ and, for all $1\leq i \neq  j\leq n$,  edges $ij$ appear independently      with probabilities $W(X_i,X_j)$ conditional to the given $X_1,\ldots,X_n$. The simplest and  most studied case is when $W$ is identical to $p$,  where $\M(n,W)$ is equivalent to  $\M(n,p)$.
	When we say the chromatic number of a graphon $W$, we implicitly mean  the chromatic number of the  exchangeable random graph constructed this way.
	
	The main focus of this paper is  asymptotics of $\chi(\G)$, where $\G \sim \M(n,W)$.
	Despite  rich literature on the study of the behaviour of the chromatic number in  various random graph models,  only a few results on colouring  non-constant graphons   are available in the literature.
	Bhattacharya, Diaconis and Mukherjee \cite{BDM2017}  characterised    the limiting distribution of the
	number of monochromatic edges in random colourings for all converging sequences of dense graphs.
	Hladk\'y and Rocha~\cite[Theorem~1.4]{Hladky2020} proved that the chromatic number of a graphon  is not upper continuous but is lower 	semicontinuous.

	Martinsson,  Panagiotou, Su, and  Truji\'c~\cite{MPSM2020} made a conjecture equivalent to  that an asymptotically optimal number of colours for $\G\sim \M(n,W)$ is achieved by taking infimum over  the  following class of colouring strategies; see Conjecture \ref{Con:main} and Remark~\ref{Rem:Martinsson}.   Given a representation 	$\sum_{i \in [k]} \alpha_k \mu_k$ of uniform measure as a finite convex combination of  probability measures on $[0,1]$, we can sample  $X_1,\ldots, X_n$, by first sampling a type from $[k]$ proportionally to $\alpha_k$ and then sampling the point  according to $\mu_k$.  This splits the exchangeable random graph $\G$   into $k$ graphs $\G_1,\ldots, \G_k$ according to the chosen type. Then, for each $i \in [k]$,  we colour $\G_i$ in a "balanced" way by using largest independent sets \emph{aligned} with the measure $\mu_i$.
In this paper we confirm  the existence of an asymptotically optimal strategy within this class   for  graphons that can be approximated by block graphons in $\mathcal{L}^\infty$-norm and for block-increasing  graphons.

	We discuss in detail our main results and the key  ideas in Section \ref{S:dense}, while most of the proofs are given in later sections.
	In Section \ref{S:SBM}, we determine the chromatic number for block graphons, using the recent result \cite[Theorem 2.1]{IsaevKang-SBM} on the chromatic number in the stochastic block model. Finally, we consider approximations of general graphons with block graphons  in Section \ref{S:general}.
	


	
	\section{Our results and the key ideas}\label{S:dense}

	Throughout the paper,  we use the standard Landau notations for asymptotic orders and
	all asymptotics are taken as $n\rightarrow \infty$ (unless it is specified otherwise).  
	We say that an event holds with high probability ({\em whp} for short)
	if the probability that it holds tends to 1 as $n\rightarrow \infty$. 
	The statements about asymptotics like   ``$\Y(n) = (1+o(1)) \Z(n)$  holds whp as $n \rightarrow \infty$''   should be interpreted as follows:  for any fixed
	$\epsilon>0$,   the event that 
	$ (1-\epsilon)\Z(n)\leq \Y(n) \leq (1+\epsilon)\Z(n)$   holds whp. 
	We use $\log$ to denote the natural logarithm.

	\subsection{"Balanced" colourings}\label{S:c_balanced}
	
	As motivations for why certain quantities play important roles in the chromatic number of graphons, we begin with some heuristics. The starting point   is the well-known result by Bollob\'as on the typical value of the chromatic number for dense case: Rewriting \eqref{Bollobas-original} using the natural logarithm, we have that for a constant $p\in(0,1)$, whp 
	\begin{equation}\label{Bollobas-natural-log}
		\chi(\G(n,p)) = (1+o(1))   \dfrac{n  \log  \dfrac{1}{1-p}  }{2 \log n}.
	\end{equation}
	One of the main ingredients to prove this result is the concentration result on the size of a largest independent set of $\G(n,p)$: whp 
	\begin{equation}\label{Bollobas-indep}
		\alpha(\G(n,p)) = (1+o(1)) \dfrac{2 \log n}{\log   \frac{1}{1-p} }.
	\end{equation}
	
	Inspired by the classic colouring strategy for $\G(n,p)$, we provide heuristics to estimate the number of colours required to properly colour  $\G \sim  \M(n,W)$ using  \emph{"balanced"} colour classes. 
	Informally speaking,  a  "balanced" set means that its empirical measure approximates the uniform measure. 
	Our naive attempt  relies on the first moment calculations for  the size of largest "balanced" independent set in $\G$.   
	
	To be more precise, let  $\lambda$ denote the uniform measure on $[0,1]$, that is, for all 
	Lebesgue measurable sets $S\subseteq [0,1]$, 
	\[
	\lambda(S):= \int_{S} 1 \ dx.
	\]
	Let $X_1,\ldots, X_n$ be the points sampled independently uniformly from $[0,1]$ for the construction of $\G \sim \M(n,W)$.
	A "balanced" set $T \subseteq[n]$ means that the empirical   measure defined by the points  $X_i$  with  $i \in T$  is close to the uniform measure:
	\[
	\lambda(S)\approx \hat{\mu}\left[\left(X_i\right)_{i \in T}\right] (S)
	:=\dfrac{| \{i \in T \st X_i \in S\}|}{|S|}
	\]
	for any measurable $S\subseteq [0,1]$.
	

	For any $T \subseteq [n]$ and given $(X_i)_{i\in T}$, the probability that $T$ is   independent   in $\G$ is given by
	\begin{align*}
		\Pr (\G[T] \text{ has no edges} \mid (X_i)_{i\in T}) = 
		\prod_{i,j\in T \st i<j}  (1-W(X_i, X_j)).
	\end{align*} 
	If $T$ is "balanced" then 
	\[
	\dfrac{1}{|T|^2}\sum_{i,j\in T \st i<j} \log (1-W(X_i, X_j)) \approx  \dfrac12 
	\int_{0}^1\int_0^1\log \left(1-W(x,y)\right)  dx \, dy.
	\]
	Note also that  a set $T$ chosen uniformly at random  from   all sets of size $t$ is typically "balanced". Thus, if $Y_t$ is the number of independent "balanced" sets of size $t$ in $\G$ then, we get,
	\[
	\E[Y_t] \approx \binom{n}{t} \exp\left(\frac{t^2}{2} \int_{0}^1\int_0^1\log \left(1-W(x,y)\right)  dx\,  dy\right).
	\]
	The expectation threshold $t^*$ can be obtained by solving $\E[Y_t]  =1$, which gives
	\[
	t^*\approx  \frac{2\log n }{ \int_{0}^1\int_0^1\log \left(\frac{1}{1-W(x,y)}\right)  dx \,dy }.
	\]

	However, the expectation threshold $t^*$  for $Y_t$ can be  significantly above the {\em existence} threshold due to the fact that our random graph model is inhomogeneous: a very dense part of $\G$ can further 
	restrict the size of a "balanced" independent set. Therefore, we need a refined version of the expectation threshold.  Note that for any $S \subseteq[0,1]$ of positive measure,  there  are approximately $\lambda(S) n$ points $X_i \in S$ and any "balanced"   set $T$ contains approximately $\lambda(S)|T|$ of such $i$. 
	Similarly to  $\E[Y_t]$, the expected number of  ways to choose $\lambda(S)t$   vertices  
	from $\lambda(S)n$ vertices corresponding $X_i\in S$ such that they form an independent set is  approximately 
	\[
	\binom{\lambda(S)n}{\lambda(S) t} 
	\exp\left(\dfrac{(\lambda(S) t)^2}{2} \int_S\int_S\log \left(1-W(x,y)\right)  dx \,dy\right).
	\]
	Then, by setting this quantity to be equal to one, we  obtain the following refined version of the expectation  threshold:  
	\[
	t_S^* \approx  \frac{2\log n }{ \frac{1}{\lambda(S)}\int_S\int_S\log \left(\frac{1}{1-W(x,y)}\right)  dx\, dy}.
	\]

	From the expression for $t_S^*$ one can see that the strongest restriction on the size of a "balanced" independent set comes from a set $S \subseteq [0,1]$ that maximises $\frac{1}{\lambda(S)}\int_S\int_S\log \left(\frac{1}{1-W(x,y)}\right)  dx\, dy$.   This motivates the following definition.
	For a graphon $W: [0,1]^2 \rightarrow [0,1)$, such that 
	$\log \left(\dfrac{1}{1-W(x,y)}\right)$ is integrable on $[0,1]^2$, let 
	\begin{equation} \label{def:var_uni} 
		\varphi(W) := \sup_{S\subseteq [0,1] \st \lambda(S)> 0} \   \dfrac{1}{\lambda(S)}\
		\int_{S } \int_S\log \left(\dfrac{1}{1-W(x,y)}\right)  dx \,dy,	
	\end{equation}
	where the supremum is taken over all measurable sets $S \subseteq [0,1]$ of  positive measure. 
	Returning to our heuristics,   the size of a largest "balanced" independent set is approximately $\frac{2\log n} {\varphi(W)}$
	(at least can not exceed this value). 
	Thus, informally, 
	$\varphi(W)$  corresponds to
	the number of colours required to properly colour  $\G \sim  \M(n,W)$ using  \emph{"balanced"} colour classes up to a  factor $\frac{n}{2\log n}$.

	Our first result shows that $  \varphi(W)\cdot \dfrac{n }{2\log n}$ is indeed an upper bound for the chromatic number of $\G\sim \M(n,W)$ for any    Riemann  integrable  graphon $W$.
	
	\begin{thm}\label{T:New1}
		Assume that $W$ is a Riemann  integrable graphon 
		and
		$\sup_{x,y \in [0,1]} W(x,y)<1$. 
		Then, whp as $n \rightarrow \infty$, 
		\begin{equation}\label{eq:balanced} 
			\chi(\G) \leq (\varphi(W)+o(1))\dfrac{n }{2\log n},  \quad \text{where}  \  \    \G \sim \M(n,W).
		\end{equation}
	\end{thm}
	We prove Theorem \ref{T:New1} in Section \ref{S:New1}. 
	In certain cases, the "balanced" strategy is optimal; see Theorem \ref{thm:balanced} stated later in this section.
	 However, for arbitrary graphons, it is naive to expect that the bound of \eqref{eq:balanced} is tight, so more advanced colouring strategy is required.

\subsection{Combinations of "balanced" colourings}\label{S:combi}
Our next idea is to extend the "balanced" colouring strategy to $\G \sim \M(n,W)$, by allowing colour classes of different types: one class might be smaller and contain  more vertices of high degrees, while another might be larger and consist mostly of vertices of small degrees. In general, a "type" is associated with a probability measure   representing the distribution of  sampled points from $[0,1]$ corresponding to  the colour class.


%

To be more formal, let $\calP(\varOmega)$ denote the set of probability measures on $\varOmega$. 
Let us represent  the uniform measure $\lambda$ as a finite convex combination  of  probability measures $\mu \in \calM \subseteq \calP([0,1])$:
\begin{equation}\label{def:convex}
\sum_{\mu \in \calM} \alpha_\mu    =1 \quad \text{ and } \quad   \sum_{\mu \in \calM} \alpha_\mu \mu = \lambda.
\end{equation}
Let  $\calC_{\lambda}$   denote  the set of all representations of  the uniform measure $\lambda$ 
as a finite convex combination of probability measures: 
\[
\calC_{\lambda}:= \Big\{(\calM, \alphavec) \st \text{$\calM  \subset \calP([0,1])$ is finite and  $\alphavec \in [0,1]^{|\calM|}$ satisfies \eqref{def:convex}} \Big\}.
\] 
Then, for any $(\calM,\alphavec) \in \calC_{\lambda}$, we can think of  $\G \sim \M(n,W)$ as a vertex-disjoint  union 
\begin{equation}\label{disj_union}
\G = \bigcup_{\mu \in \calM} \G_{\mu},
\end{equation}
where the random graphs $\G_{\mu}$ are constructed as follows.   For each $i \in [n]$, we first pick the measure $\mu$ with probability $\alpha_\mu$, then we sample $X_i$ according $\mu$.  
By our construction and  using \eqref{def:convex},
the points $X_1,\ldots, X_n$ are independent and distributed uniformly  on $[0,1]$. We define  $\G_{\mu}$   to be the  induced subgraph of $\G$ corresponding to the points sampled according $\mu$.

The idea is to apply the "balanced" colouring strategy from the previous section to colour each $\G_{\mu}$. But, in order to do so, we should incorporate the measure $\mu$ in the formula for required number of colours. This leads to the following quantity similar to \eqref{def:var_uni}:
\begin{equation}
\varphi(\mu,W) := \sup_{S\subseteq [0,1] \st \mu(S)> 0} \   \dfrac{1}{\mu(S)}\
\int_{S}\int_{S}	 \log \left(\dfrac{1}{1-W(x,y)}\right)  d\mu(x)d\mu(y),	 \label{def:var} 
\end{equation}
where the supremum is taken over all measurable sets $S \subseteq [0,1]$ with positive measure $\mu(S)$.    
In fact, one can show
\[
\G_\mu \sim \calG(n_{\mu}, W_{\mu}) \qquad  \text{and}\qquad
\qquad \varphi(\mu,W) = \varphi(W_{\mu}),
\] 
where $n_{\mu}$ is the (random) number of vertices from $[n]$ sampled according to $\mu$, and  $W_{\mu}$ is a graphon corresponding to $W$ after a certain  transformation   depending on $\mu$.
Using this idea, 
we derive the following result as a consequence of Theorem \ref{T:New1}.
\begin{thm}\label{C:main2}
Assume that $W$ is a Riemann  integrable graphon 
and $\sup_{x,y \in [0,1]} W(x,y)<1$.
Then, for any  $(\calM,\alpha) \in \calC_\lambda$  whp as $n \rightarrow \infty$, 
\begin{equation*} 
	\chi(\G) \leq  \Big(\sum_{\mu \in \calM} \alpha_{\mu} \varphi(\mu, W)  + o(1)
	\Big) \dfrac{n}{2\log n}, \quad \text{where}  \  \    \G \sim \M(n,W).
\end{equation*}
\end{thm}
Details or the proof  of Theorem \ref{C:main2}  are given in  Section \ref{S:main2}.

\subsection{An example}
As an illustration of  various colouring strategies discussed in Sections~\ref{S:c_balanced} and~\ref{S:combi}, we consider the graphon $W$ displayed in  Figure~\ref{f:block}, which is defined as follows:
\[
W(x,y) = 
\begin{cases}
\dfrac34, & \text{when $(x,y)\in (S_1\times S_3) \cup (S_2\times S_2) \cup (S_3 \times S_1)$,}\\
\dfrac78, & \text{when $(x,y) \in S_3\times S_3$,}\\
\dfrac12, & \text{otherwise,}
\end{cases}
\]
where
\[
S_1 = [0,\tfrac13),  \qquad S_2:=[\tfrac13,\tfrac23), \qquad S_3:=[\tfrac23,1].
\] 

\begin{figure}[h!]
\begin{minipage}[r]{0.45\textwidth}
	\flushright
	\begin{tikzpicture}
		\begin{scope}[scale=0.8]
			
			\draw [fill=lightgray!50!white] (0,0) rectangle (2,2);
			\node   at (1,1) {$\dfrac12$};
			\draw [fill=lightgray!50!white] (2,0) rectangle (4,2);
			\node  at (3,1) {$\dfrac12$};
			\draw [fill=lightgray!50!white] (2,0) rectangle (0,-2);
			\node  at (1,-1) {$\dfrac12$};
			\draw [fill=gray!80!white] (2,0) rectangle (4,-2);
			\node   at (3,-1) {$\dfrac34$};
			\draw [fill=lightgray!50!white] (4,-2) rectangle (6,0);
			\node   at (5,-1) {$\dfrac12$};
			\draw [fill=gray!80!white] (4,0) rectangle (6,2);
			\node   at (5,1) {$\dfrac34$};
			
			\draw [fill=gray!80!black] (6,-4) rectangle (4,-2);
			\node   at (5,-3) {$\dfrac78$};
			
			\draw [fill=lightgray!50!white] (2,-4) rectangle (4,-2);
			\node   at (3,-3) {$\dfrac12$};
			
			\draw [fill=gray!80!white] (0,-2) rectangle (2,-4);
			\node   at (1,-3) {$\dfrac{3}{4}$};
			
		\end{scope}
	\end{tikzpicture}
\end{minipage} 
\
\begin{minipage}[c]{0.41\textwidth}
	\centering
	\begin{itemize}
		\item[] "Balanced" strategy  requires $\left(\dfrac{7\log 2 }{9} +o(1)\right)\dfrac{n}{\log n} $ colours. 
		\item[]  "Greedy" strategy  requires $\left(\dfrac{5 \log 2}{6}+o(1)\right)\dfrac{n}{\log n}  $ colours. 
		\item[] Optimal colouring
		requires $\left(\dfrac{(5+\sqrt{3})\log 2}{9} +o(1)\right)\dfrac{n}{ \log n} $ colours.
	\end{itemize}
\end{minipage}
\caption{ A graphon $W$ and colouring strategies for $\calG(n,W)$}
\label{f:block} 
\end{figure}

As explained in Section  \ref{S:c_balanced}, the "balanced" strategy considers colour classes  that have the same number of vertices from each block $S_i$ and $(\varphi(W) +o(1))\dfrac{n}{2\log n}$ colours.  One can  show that 
\begin{align*}
\varphi(W)  &=  \int_{[0,1] } \int_{[0,1]}\log \left(\dfrac{1}{1-W(x,y)}\right)  dx \,dy
\\&=   \dfrac{5}{9} \log 2 + \dfrac{1}{3} \log 4 + \dfrac{1}{9} \log 8 = \dfrac{14}{9} \log 2.
\end{align*}
Thus, the "balanced" strategy uses
\[
\left(\dfrac{7\log 2 }{9} +o(1)\right)\dfrac{n}{\log n} \approx 0.54 \dfrac{n}{\log n}
\]
colours. 

Alternatively, we can try the "greedy" strategy. First, we pick largest independent sets, which consist only of vertices from $S_1$. Once we run out  of such sets, we have to use vertices from $S_2$ and $S_3$. Using the first moment calculations similar to Section  \ref{S:c_balanced}, one can show that among such independent sets, the largest size is achieved when we take twice more vertices  from $S_2$ than from $S_3$.
After colouring all such sets, it remains to colour half of vertices from $S_3$. Thus, the  "greedy" strategy corresponds to the following representation of the uniform measure $\lambda$:
\[
\lambda = \dfrac13 \mu_1 + \dfrac12 \mu_{23} + \dfrac16 \mu_3,
\]
where $\mu_1$ and $\mu_3$ are the uniform measures on $S_1$ and $S_3$, respectively,  and the measure $\mu_{23}$ is defined by
\[
\mu_{23} (U) = 2 \lambda(U \cap S_2) + \lambda(U \cap S_3). 
\]
One can compute  that
\begin{align*}
\varphi(\mu_1, W) = \log 2,
\qquad 
\varphi(\mu_{23}, W) =     
\dfrac53 \log 2,
\qquad 
\varphi(\mu_3, W) =    3 \log 2.
\end{align*}
We conclude  that the "greedy" strategy uses 
\[
\left( \dfrac 13 \log 2 + \dfrac12 \cdot \dfrac53 \log 2 + \dfrac16 \cdot 3\log 2 +o(1)\right) \dfrac{n}{2\log n} = \left(\dfrac{5}{6} \log 2 + o(1)\right)\dfrac{n}{\log n} \approx 0.58 \dfrac{n}{\log n}
\]
colours. In particular, we get that the "balanced" strategy  works better than "greedy" strategy for this graphon $W$.

In Section \ref{S:example}, we show that  the best colouring strategy have two types of independent sets. One type uses vertices from $S_1$ and $S_2$ in proportion $1:r$ (with more vertices from $S_1$), another uses vertices from $S_3$ and $S_2$ in proportion $(1-r):1$ (with more vertices from $S_3$), where $r = \dfrac{3\sqrt 3 -5}{2}$.  This   strategy  corresponds   to the following representation of the uniform measure $\lambda$:
\[
\lambda = \dfrac{\sqrt{3}-1}{2}\mu_{12}' + \dfrac{3-\sqrt{3}}{2}\mu_{23}',
\]
where the measures $\mu_{12}'$ and $\mu_{23}'$ are defined by
\begin{align*}
\mu_{12}'
(U) &=  (1+\sqrt 3) \lambda(U \cap S_1) + (2-\sqrt{3})\lambda(U \cap S_2),\\
\mu_{23}' (U) &= \dfrac{6-\sqrt{3}}{3} \lambda(U \cap S_2) + \dfrac{3+\sqrt{3}}{3}\lambda(U \cap S_3).
\end{align*}
Then, in Section \ref{S:example}, we compute that 
\begin{align*}
\varphi(\mu_{12}', W) &=  \dfrac{(1+\sqrt{3})^2}{9} \log 2 + \dfrac{2(1+\sqrt{3})(2-\sqrt{3})}{9} \log 2 + \dfrac{(2-\sqrt{3})^2}{9} \log 4
= \dfrac{16-4\sqrt 3}{9} \log 2,\\
\varphi(\mu_{23}', W) &=  \dfrac{(6-\sqrt{3})^2}{9^2} \log 4 + \dfrac{2(6-\sqrt{3})(3+\sqrt{3})}{9^2} \log 2 + \dfrac{(3+\sqrt{3})^2}{9^2} \log 8 = \dfrac{16}{9} \log 2.
\end{align*}
Thus, this strategy uses
\[
\left(\dfrac{\sqrt{3}-1}{2} \cdot   \dfrac{16-4\sqrt 3}{9} \log 2 + 
\dfrac{3-\sqrt{3}}{2} \cdot\dfrac{16}{9} \log 2 + o(1)\right) \dfrac{n}{2\log n} 
= \left(\dfrac{(5+\sqrt 3)}{9} \log 2 +o(1)\right)\dfrac{n}{\log n} \approx 0.52 \dfrac{n}{\log n}
\]
colours.
The computations and the proof that  this  colouring strategy is optimal  are given in detail in Section \ref{S:example}.  In particular, we can see that it outperforms the "balanced" strategy and  the "greedy" strategy.

\subsection{The conjecture}
Observe that the asymptotically tightest possible bound in  Corollary \ref{C:main2} is achieved when we take the smallest possible value of $\sum_{\mu \in \calM} \alpha_\mu \varphi(\mu,W)$, leading to the following definition:
\begin{equation} \label{def:var-star} 
\varphi_*(W) := \inf _{(\calM, \alphavec) \in \calC_{\lambda}} \   \sum\nolimits_{\mu \in \calM} \alpha_\mu \varphi(\mu,W).
\end{equation}
Equivalently,   we have 
\[
\varphi_*(W) = \inf\nolimits_{\xi}\ \E   \varphi(\mu^{\xi},W),
\]
where the infimum is taken over all $\xi \in\calP(\calP([0,1]))$ such that 
$ \E  \mu^{\xi} = \lambda$, where  $\mu^{\xi} \in \calP([0,1])$ is  distributed according to $\xi$.
Furthermore, one can show that the infimum over $\xi$ is  achieved   since the space of probability measures on probability measures on a compact set is compact. With the refined quantity $\varphi_*(W) $,   we obtain an upper bound of the chromatic number of $\G$, 
\begin{equation}\label{eq:refined}
\chi(\G) \leq  \big(\varphi_*(W) + o(1)
\big)\ \frac{n}{2\log n},
\end{equation}
immediately from Theorem \ref{C:main2}.

The main focus of the paper is answering the natural question whether  \eqref{eq:refined} is tight or not, that is, whether a {\em finite} number of different types of colour  classes   is sufficient to approximate  the chromatic number up to an  additive error of order $o(\frac{n}{\log n})$.
This question appears for the first time  in the recent work by Martinsson,  Panagiotou, Su, and~Truji\'c~\cite{MPSM2020}.  A variant of \cite[Conjecture~1.2]{MPSM2020} is stated below.
\begin{conj}\label{Con:main}
Assume that  $W$  is a graphon such that
\begin{equation}\label{eq:Winfsup}
	0<\inf_{x,y \in [0,1]} W(x,y) \leq \sup_{x,y \in [0,1]} W(x,y)<1.
\end{equation}
Then, whp as $n \rightarrow \infty$, 
\begin{equation}\label{eq:con}
	\chi(\G) =  (1+o(1))\, \varphi_* ( W) \frac{n}{2\log n} ,  \quad \text{where}  \  \    \G \sim \M(n,  W).
\end{equation}
\end{conj}

\begin{remark}\label{Rem:Martinsson}
In fact, \cite[Conjecture~1.2]{MPSM2020}  only assumes    that the
essential infimum of $W$   lies in $(0,1)$.  
However,   we certainly need an upper bound on  $W$ since otherwise there might  exist a clique of  linear size
in $\G \sim  \M(n, W) $.
Also,    \cite[Conjecture~1.2]{MPSM2020}  is  stated in a slightly different form
$\chi(\G) = (1+o(1))\displaystyle \dfrac{n}{c^* \log n}$, where definition of $c^*$ is quite technical so we omit it here. Using the Radon-Nikodym theorem, one can establish that $\varphi_*( W)    =2 /c^*$, that is, the conjectured formulas are equivalent.
\end{remark}  

%

From the results presented above, the top half of  Conjecture \ref{Con:main} is straightforward.
\begin{cor}\label{Cor:upper}
	The  upper bound for $\chi(\G)$ in Conjecture \ref{Con:main} holds
	for any  Riemann  integrable graphon  $W$.
\end{cor}
\begin{proof}
	By definitions \eqref{def:var} and \eqref{def:var-star}, it is straightforward to show that  
	\begin{equation}\label{W-trivial}
		\log \left( \frac{1}{1-\inf_{x,y \in [0,1]} W(x,y) }\right) \leq \varphi_*(W)  \leq \varphi(\mu, W)  \leq	\log \left( \frac{1}{1-\sup_{x,y \in [0,1]} W(x,y) }\right). 
	\end{equation}
	From \eqref{eq:Winfsup} and \eqref{W-trivial}, we find that $\varphi_*(W) = \Theta(1)$. Using    \eqref{eq:refined}, we get
	\[
	\chi(\G) \leq (\varphi_* ( W) +o(1))\frac{n}{2\log n} = (1+o(1))\, \varphi_* ( W) \frac{n}{2\log n},
	\]
	as required. 
\end{proof}

To our knowledge, the   lower bound for  $\chi(\G)$ in Conjecture \ref{Con:main} for arbitrary graphons remains open.  Moreover, we lack a result answering the following question in full generality.      
\begin{equation}\label{question}
	\text{Is it true that $\dfrac{\log n}{n}\E [\chi(\G)]$ converges as $n \rightarrow \infty$, where $\G\sim \calG(n,W)$?}
\end{equation}
Clearly, Conjecture \ref{Con:main} implies the positive answer.  It would be interesting to find a simple argument justifying the existence of limit without going into tedious calculations of its value.

  In this paper, we confirm  Conjecture \ref{Con:main}  for block graphons and certain classes of graphons that can be approximated by block graphons; see the next three subsections for details.

%
%
%

\subsection{Block graphons}

A graphon  $W$ is called   a \emph{block graphon}  if there exists a finite family $\calF$   of   disjoint measurable sets partitioning $[0,1]$  such that  $W$ is constant on  $S \times S'$ for any $S,S' \in \calF$. Specifically, we say $W$ is a \emph{$k$-block graphon} if
the number of such sets equals $k$, that is,
\[
W   \equiv \sum_{i,j \in [k]}  p_{ij}  \one_{S_i \times S_j}, 
\]
where $p_{ij}\in [0,1]$ for all $i,j \in [k]$, 
the sets $(S_i)_{i\in [k]}$ form a partition of $[0,1]$, and
$\one_{S_i \times S_j}$ is the characteristic function of  $S_i \times S_j$.
The graphon displayed in Figure \ref{f:block} is an example.

For a graphon $W:[0,1]^2 \rightarrow [0,1)$, let 
\begin{equation} \label{def:var-k} 
	\varphi_k(W) := \inf _{\substack{(\calM, \alphavec) \in \calC_{\lambda}\\
			|\calM| \leq k }} \   \sum\nolimits_{\mu \in \calM} \alpha_\mu \varphi(\mu,W).
\end{equation}
From definitions \eqref{def:var_uni} and \eqref{def:var-star}, observe that
\[
\varphi_1(W) = \varphi(W) \quad \text{ and } \quad \varphi_k(W) \searrow \varphi_{*}(W) \text{ as $k \rightarrow \infty$,}
\]
where $ \searrow$ means  decreasing convergence.
The quantity $\varphi_k(W) \frac{n}{2\log n}$ corresponds to the number of colours required for proper colouring of $\G \sim \calG(n,W)$ with at most $k$ different types of colour classes. 
We show that if $W$ is a $k$-block graphon, then  such colouring strategy is asymptotically optimal; see the following theorem. 

\begin{thm}\label{T:block}
	If  $W:[0,1]^2 \rightarrow [0,1)$ is a $k$-block graphon,  then whp as $n\rightarrow \infty$  
	\[
	\chi(\G) = \left(\varphi_k(W) + o(1)\right) \dfrac{n}{2\log n}, \qquad \text{where $\G\sim \calG(n,W)$}.
	\]
\end{thm}

In fact, in Section \ref{S:SBM}, we prove a more general result, Theorem \ref{T:sequences1},  that immediately implies Theorem \ref{T:block}.
As a  corollary, we get the following.

\begin{cor}\label{C:block}
	Conjecture \ref{Con:main} holds if $W$ is a block graphon.
\end{cor}
\begin{proof}
	Assume that $W$ is a $k$-block graphon. 
	Combining Theorems~\ref{C:main2} and~\ref{T:block}, we get that whp
	\[
	\left(\varphi_k(W) + o(1)\right) \dfrac{n}{2\log n} =    \chi(\G)    \leq (\varphi_*(W) + o(1)) \dfrac{n}{2\log n}.
	\]
	Recalling that $\varphi_k(W)\geq \varphi_*(W)$ and repeating the arguments of Corollary \ref{Cor:upper},  we establish that $\varphi_k(W) = \varphi_*(W) = \Omega(1)$ and complete the proof.
\end{proof}

In general, the computation of $\varphi_k(W)$ for a $k$-block graphon $W    \equiv \sum_{i,j \in [k]}  p_{ij}  \one_{S_i \times S_j}$
is a  continuous optimisation problem in $k^2-k$ dimensions. Indeed,
the distribution of $\G_\mu \sim \calG(n,W^{\mu})$ only depends on 
probabilities $\mu(S_i)$ which represent  $k$ variables for each $\mu$.  Since 
$\sum_{i\in k} \mu(S_i) =1$ and $|\calM| \leq k$, overall we get $k^2- k$   variables.
An example of such computation for the graphon from Figure \ref{f:block} is given in Section~\ref{S:example}.

Next results classifies all block graphons $W$  for which the "balanced" colouring strategy is asymptotically optimal.
\begin{thm}\label{thm:balanced}
	Let $
	W   \equiv \sum_{i,j \in [k]}  p_{ij}\one_{S_i \times S_j} $ be a 
	$k$-block graphon. Then $\varphi_*(W) = \varphi(W)$ if and only if 
	there is a set $U \subseteq [k]$ such that
	\[
	\varphi(W) = \varphi(\lambda^U,W), 
	\]
	where $\lambda^U$ is the uniform measure on $\bigcup_{i\in U} S_i$, 
	and 
	for all $\yvec = (y_1,\ldots,y_k)^T\in \Reals^k$ with $\sum_{i \in U} y_i =0$ and $y_i = 0$ for $i\notin U$, we have
	\[
	\sum y_i y_j \log \dfrac{1}{1-p_{ij}} \geq 0.
	\]
\end{thm}
In particular, Theorem \ref{thm:balanced} implies that $\varphi_*(W) = \varphi(W)$ if the matrix 
\begin{equation}\label{defQW}
Q(W):=\left(\log \dfrac{1}{1-p_{ij}}\right)_{i,j\in [k]}
\end{equation}
 is positive semidefinite on the subspace defined by $\sum_{ i \in [k]}y_i =0$. 
Below, we  consider two particular examples of this situation, for which we can give  explicit formulas for $\chi(\G)$.  
\begin{cor}\label{C:block1}
	Let $(S_i)_{i\in [k]}$ be a partition of $[0,1]$ into $k$ intervals of lengths
	$
	\ell_1 \geq \ell_2 \geq \cdots \geq \ell_k. 
	$
	Let $p\leq p_0$  be numbers from $(0,1)$  and  
	\[
	W(x,y) := 
	\begin{cases}
		p_0, & \text{if $x,y \in S_i$  for some $i \in [k]$,}\\
		p, & \text{otherwise.} 
	\end{cases}
	\]
	If   $\G\sim \calG(n,W)$, then whp
	\[
	\chi(\G) = \left(1+ o(1)\right) 
	\max_{i \in [k]}\left\{ \dfrac{\sum_{j  \in [i]} \ell_j^2}{\sum_{j \in [i]} \ell_j} 
	\log \dfrac{1-p}{1-p_0}+  \sum_{j \in [i]} \ell_j \log\dfrac {1}{1-p} \right\}\dfrac{n}{2\log n}.
	\]
\end{cor}

\begin{cor}\label{C:block2}
	Let $(S_i)_{i\in [k]}$ be a partition of $[0,1]$ into $k$ intervals of equal length $1/k$. Let $p_1 \geq p_2 \geq \cdots \geq p_k$ be numbers from $(0,1)$, and $p \in [0,p_k]$, and 
	\[
	W(x,y) := 
	\begin{cases}
		p_i, & \text{if $x,y \in S_i$  for some $i \in [k]$,}\\
		p, & \text{otherwise.} 
	\end{cases}
	\]
	If   $\G\sim \calG(n,W)$, then whp
	\[	\chi(\G) = \left(1+ o(1)\right) 
	\max_{i \in [k]}\left\{    \dfrac{i-1}{k}\log\dfrac {1}{1-p}  +  \dfrac{1}{ik} \sum_{j \leq i}    \log\dfrac{1}{1-p_j}  \right\}\dfrac{n}{2\log n}.
	\]
\end{cor}

We prove Theorem \ref{thm:balanced} and Corollaries \ref{C:block1} and~\ref{C:block2} in Section \ref{S:balanced}.

	\subsection{Approximations with block graphons} 
 
A graphon $W$ is \emph{block-regulated} if and only if 
	for any $\delta>0$ there exists a  block  graphon $W^{\delta}$ such that 
	\[
  \sup\limits_{x,y \in [0,1]}|W(x,y) - W^{\delta}(x,y)| \leq \delta.
	\]
	The notion of regulated functions is well known for the  case of single variable and is equivalent to the existence of one-sided limits.  The class of block-regulated graphons defined above is a special case of the generalisation of this notion to $\Reals^d$ proposed by Davison \cite{Davison1979}.  In particular, one can show that any continuous graphon $W$ is block-regulated.

	\begin{thm}\label{T:block-regulated}
		Conjecture \ref{Con:main} holds if $W$ is   block-regulated. 
	\end{thm}

	We prove Theorem \ref{T:block-regulated} in Section \ref{S:b-regulated}. 
	
	\begin{remark}
	In this paper we consider approximations by  block graphons in $\calL^\infty$-norm.  
	It is tempting to believe that a similar approach should  work for arbitrary graphons. For example, it is well known that any graphon can be approximated by block graphons
	in a weaker norm, known as the cut-norm; see,  for example,    Lov\'asz \cite[Section 9.2]{Lovasz2012}.  
	The complication is that  the chromatic number is quite sensitive with respect to the  cut-norm. For example,  a planted clique of sublinear size can change the chromatic number dramatically, while the graph limit remains the same.
	\end{remark}

	  To go beyond step-regulated graphons, 
	  we consider two graphons $W_L$ and $W_R$ displayed in Figure \ref{f:non-block} that are defined by
	  \[
	      W_L(x,y):= \begin{cases}
	      							\frac34 &\text{if  $x+y \leq 1$,}\\
	      							\frac12 &\text{otherwise};
	      					\end{cases}
	\qquad \quad
		 W_R(x,y):= \begin{cases}
			\frac34 &\text{if $x,y  \leq \frac12$,}\\
			\frac12 &\text{otherwise}.
		\end{cases}
	\]
	  
	  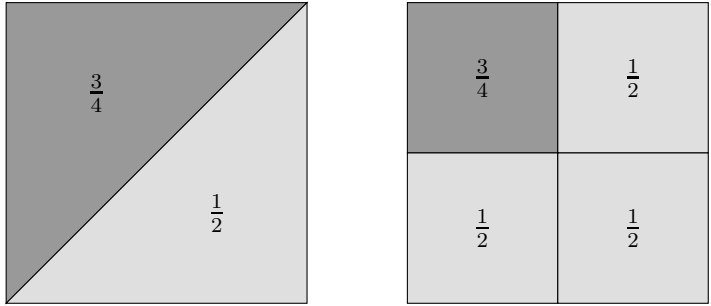
\begin{figure}[h!]
	  		\begin{minipage}[l]{0.45\textwidth}
	  		\flushright
	  		\begin{tikzpicture}
	  			\begin{scope}[scale=1]

	  				\filldraw[draw=black, fill=gray!80!white] (4,2) -- (0,2) -- (0,-2)-- cycle;
	  				\filldraw[draw=black, fill=lightgray!50!white] (4,2) -- (4,-2) -- (0,-2)-- cycle;
	  				
	  				\node   at (1.2,0.8) {$\dfrac34$};
	  				
	  				\node   at (2.8,-0.8) {$\dfrac12$};
	  				%
	  				%
	  				%
	  				
	  			\end{scope}
	  		\end{tikzpicture}
	  	\end{minipage} 
	  	\begin{minipage}[l]{0.4\textwidth}
	  		\centering
	  		\begin{tikzpicture}
	  			\begin{scope}[scale=1]
	  				
	  				\draw [fill=gray!80!white] (0,0) rectangle (2,2);
	  				\node   at (1,1) {$\dfrac34$};
	  				\draw [fill=lightgray!50!white] (2,0) rectangle (4,2);
	  				\node  at (3,1) {$\dfrac12$};
	  				\draw [fill=lightgray!50!white] (2,0) rectangle (0,-2);
	  				\node  at (1,-1) {$\dfrac12$};
	  				\draw [fill=lightgray!50!white] (2,0) rectangle (4,-2);
	  				\node   at (3,-1) {$\dfrac12$};

	  			\end{scope}
	  		\end{tikzpicture}
	  	\end{minipage}
	  	\caption{ Graphons $W_L$ and $W_R$ with  the chromatic number  $\left(\dfrac{5}{8}\log 2+o(1)\right) \dfrac{n}{\log n}$. }
	  	\label{f:non-block} 
	  \end{figure}

	  The left graphon $W_L$ is one of   the simplest examples   of a graphon which is not block-regulated.   
	 Thus,  Theorem \ref{T:block-regulated} does not apply.
	  Nevertheless,   it is not difficult to show that Conjecture~\ref{Con:main} holds for this graphon.   
	   Since $W_L \geq W_R$ there is a coupling $(G_L,G_R)$
	   such that $\G_L \sim \M(n,W_L)$,  $\G_R \sim \M(n,W_R)$, and $\G_R \subseteq \G_L$. Therefore, $\chi(\G_L) \geq \chi(\G_R)$. Combining Corollary \ref{C:block} and Theorem \ref{C:main2}, we conclude that 
	  if 
	   \begin{equation}\label{W-LR}
	   			\varphi_*(W_L) \leq \varphi_*(W_R)  
	   \end{equation}
	   then $\chi(\G_L)$ is asymptotically equivalent to $\chi(\G_R)$ and 
	   Conjecture~\ref{Con:main} holds for $W_L$.
	   
	   From Theorem \ref{thm:balanced} and any of its corollaries, we know that
	   \[
	       \varphi_*(W_R)   = \varphi(W)= \dfrac{3}{4}\log 2+\dfrac{1}{4} \log 4 = \dfrac{5}{4}\log 2. 
	   \]
	   For any positive integer $k$ we can consider the colouring strategy corresponding to the following representation  of the uniform measure as a convex combination 
	   \[
	   		\lambda := \sum_{i \in [k]} \dfrac{2}{2k+1} \mu_i +  \dfrac{1}{2k+1} \mu_{k+1}, 
	   \]
	   where $\mu_i$ is the uniform measure on $\left[\frac{i-1}{2k+1} ,  \dfrac{i}{2k+1}\right) \cup\left[\frac{2k-i}{2k+1} ,  \dfrac{2k+1 - i}{2k+1}\right)$ for $i \in [k]$ and $\mu_{k+1}$ is the uniform measure on $\left[\frac{2k}{2k+1} ,  1\right]$. Observing that $W_L$ restricted to  the support of $\mu_i$ is equivalent to $W_R$, we get that $\varphi(\mu_i,W_L) = \varphi(W_R)$ for each $i \in [k]$. Therefore,
	   \[
	   			\varphi_*(W_L) \leq \varphi_{k+1}(W_L) \leq 
	   			\sum_{ i \in[k]} \dfrac{2}{2k+1} \varphi(W_R)  + \dfrac{1}{2k+1} \varphi(\mu_{k+1},W_L)=
	   			 \varphi(W_R) +O(k^{-1}).
	   \]
	   Since we can take $k$ to be arbitrary large, we get \eqref{W-LR}.

	   In the example above, we observed that 
	    a colouring strategy involving at most $k$ different types of colours approximates the optimal strategy  up to  $O(k^{-1}) \frac{n}{\log n}$ colours. Our final theorem generalises this observation for  the following two classes of graphons.
	A graphon  $W$ is  \emph{block-increasing}  if there exists a finite family $\calF$   of   disjoint intervals partitioning $[0,1]$  such that  $W$   increases with respect to both coordinates  on $S \times S'$   for any $S,S' \in \calF$. 
	Similarly,  $W$ is  \emph{block-Lipschitz} if   it is a Lipschitz function within each block 
	 $S \times S'$.

	\begin{thm}\label{Thm:con-inc}
		Assume that $W$ is block-increasing or block-Lipschitz. Then, there exists a constant $c = c(W)>0$ such that for 
		any fixed $k\in \Naturals$, whp
		\[
		\chi(\G) \geq \left(\varphi_k(W)  - c/k\right) \dfrac{n}{2\log n},  \qquad \text{where $\G\sim \calG(n,W)$.}
		\]
	\end{thm}
	We prove Theorem \ref{Thm:con-inc} in Section \ref{S:b-cont}.
	Theorems \ref{C:main2} and  \ref{Thm:con-inc} immediately imply the following corollary.
	\begin{cor}
		Conjecture \ref{Con:main} holds if $W$ is  block-increasing.
	\end{cor}
	\begin{proof}
		Note that any   block-increasing graphon is Riemann integrable.
		Thus, we can use the upper bound from Theorem \ref{T:New1}. 
		Recalling $\varphi_k(W) \geq \varphi_*(W)$ and using Theorem~\ref{Thm:con-inc}, we get that whp
		\[
		(\varphi_*(W)+o(1))\dfrac{n}{2\log n} \geq  \chi(\G) \geq \left(\varphi_*(W)  - c/k\right) \dfrac{n}{2\log n}.
		\]
		Since we can take $k$ arbitrarily large, we conclude that
		$\chi(\G) =  (\varphi_*(W)+o(1))\dfrac{n}{2\log n} $ whp.
		To complete the proof, we show $\varphi_*(W) =\Omega(1)$ by repeating the arguments of Corollary \ref{Cor:upper}. 
	\end{proof}

 There is no need to state a similar result for block-Lipschitz graphons, since they are block-regulated, so this case is covered by Theorem \ref{T:block-regulated}. Beyond that, by a slight tweak of our arguments, one can show that  Conjecture \ref{Con:main} holds for a sum of block-increasing and block-regulated graphons. Also, one can apply any measure preserving transformation to further extend the class of graphons for which Conjecture \ref{Con:main} holds: for example, we can get any block-decreasing graphon by considering  $W'(x,y):=W(1-x,1-y)$. However, these tricks are not sufficient for the general case. In particular, we have no way to check Conjecture \ref{Con:main} for $W(x,y) = \frac{1}{10}+\frac12 \sin^2 \frac{1}{|x+y-1|} $ (set to be $\frac{1}{10}$  for $x+y=1$) or even justify  the existence of the limit  in \eqref{question}.

	\section{Block graphons and the stochastic block model}\label{S:SBM}
	

In addition to $\M(n, W)$, we consider another family of random graphs associated with a graphon  $W:[0,1]^2 \rightarrow [0,1]$. Given a vector $\uvec = (u_1,\ldots,u_n)^T \in [0,1]^n$,
we generate a random graph $\G^{\uvec}\sim \M(\uvec, W)$  
with vertex set $[n]$  by including  
edges $ij$,  for all  $1 \leq  i \neq j  \leq n$,   independently with probabilities $W(u_i,u_j)$.
In particular, if   the components of  $\X(n) = (X_1, \ldots, X_n)^T \in [0,1]^n$ are sampled independently uniformly at random    then $\G^{\X(n)} \sim \M(n,  W)$.

If $W$ is a block graphon then the model $\M(\uvec,W)$ is equivalent to \emph{the stochastic block model,} in which all vertices are distributed between several different blocks and the probabilities of adjacencies of vertices depend only on the block they belong to.  In Section \ref{S:SBM-chromatic}, we recall the result of \cite{IsaevKang-SBM} on the asymptotics of  the chromatic number  in this model
and, as a corollary,  prove an asymptotic formula for  $\chi(\G^{\uvec})$. 
Then, in Section \ref{S:sequences}, we establish the following result. 

\begin{thm}\label{T:sequences1}
	Let  $\eps>0$ be fixed  and   $k = k(n) \in \Naturals$ satisfy
	$k =  o (\log n)$.	
	Suppose
	a sequence $(W_n)_{n\in \Naturals}$ of   $k$-block graphons    
	satisfies
	$
	\sup_{x,y \in [0,1]} W_n(x,y) \leq 1-\eps$.  
	Then,   whp  
	\[
	\chi(\G) =   
	(\varphi_k (W_n)+o(1)) \frac{n}{2\log n} ,    \qquad \text{where $\G\sim \calG(n,W_n)$.}
	\]	
	%
\end{thm}
Considering a trivial sequence $W_n = W$ for all $n$ in  Theorem \ref{T:sequences1}, 
we immediately get Theorem~\ref{T:block}.  Furthermore, approximations with  sequences of $k$-block graphons $W_n$ with growing $k = k(n)$ is crucial  for establishing Conjecture \ref{Con:main} for more general graphons $W$; see  Section~\ref{S:general} for details.


\subsection{Stochastic block model}\label{S:SBM-chromatic}

In this section, we begin with the notations and the results  of \cite{IsaevKang-SBM}.   
The  asymptotics of the chromatic number in the stochastic block model is given  by a certain function  $w_*$ related to $\varphi_{*}(W)$.  We also discuss some properties of  $w_*$  which  are needed later in the proofs. 
%
%
%


%

For a positive integer $\nb$, a  vector $\n=(n_1,\ldots,n_{\nb})^T\in \mathbb{N}^{\nb}$, and 
a $\nb \times \nb$ symmetric matrix  $P=(p_{ij})_{i,j \in [k]}$ with $p_{ij}  \in [0,1]$,  
a  random graph $\G$ from the stochastic block model  $\SBM$,  denoted by $\G \sim \mathcal \M(n,P)$,   is constructed as follows: 
\begin{itemize}
	\item the vertex set $V(\G)$ is partitioned into $k$ disjoint blocks 
	$B_1,\ldots, B_k$  of  sizes $|B_i| = n_i$ for $i \in [k]$ (and we write $V(\G) = B_1 \cup \cdots \cup B_k$);   
	\item  each pair $\{u, v\}$ of distinct vertices $u,v \in V(\G)$ is included in the edge set $E(\G)$,  
	independently of one another,  with probability 
	$$ 
	p(u,v) := p_{ij},
	$$
	where $i = i(u) \in [k]$ and $j = j(v) \in [k]$ are such $u \in B_i$ and $v \in B_j$.
\end{itemize}
For all asymptotic notations   used  in this section, 
we implicitly consider sequences of vectors $\n= \n(n)$ and matrices $P=P(n)$, where  $n \rightarrow \infty$. Our  probability  bounds (including whp results) hold uniformly over all   sequences  $\n(n)$ and  $P(n)$   satisfying stated  assumptions where the implicit functions like in $o(\cdot)$ depend on $n$ only.

In the following, we always assume that $ p_{ij} \in [0,1)$ for all $i,j \in [k]$.
Define the $k\times k$  symmetric matrix  $Q$ by 
\begin{equation*}
	Q:= (q_{ij})_{i,j \in [k]}, \qquad \text{where}\ \ q_{ij} :=   \log \left(\dfrac{1}{1-p_{ij}}\right).
\end{equation*}
Let $\pReals :=[0,+\infty)$ and, for $\xvec,\yvec \in \Reals^k$,
we denote 
\[ 
\yvec \preceq \xvec \text{ whenever $ \xvec - \yvec \in \pReals^k$.}\]
The norm notation $\|\cdot\|$ stand for the standard $1$-norm:
\[
\|\xvec\| := |x_1| + \ldots + |x_k|.
\]

Let  $w: \pReals^k  \to  \pReals$ be defined by 
\begin{equation}\label{w_def}
	w(\xvec)  := \max_{\boldsymbol{0} \preceq \yvec \preceq \xvec}\ \dfrac{\yvec^T\, Q\, \yvec}{ \|\yvec\|},
	\qquad{\xvec \in \pReals^k},
\end{equation}
where  $\boldsymbol{0} = (0,\ldots,0)^T \in \Reals^k$ 
and $\dfrac{\boldsymbol{0}^T\, Q\, \boldsymbol{0}}{ \|\boldsymbol{0}\|}$   is taken to be zero.
Note that $\dfrac{\yvec^T\, Q\, \yvec}{ \|\yvec\|}$  is   a continuous function of $\yvec$, which  achieves its 
the maximal value on the compact set   $\{\yvec \in \pReals^k:   \yvec  \preceq  \xvec\}$. 
In fact, it is always achieved  at a corner, that is,    $y_i \in \{0,x_i\}$ for all $i\in [k]$; see Theorem \ref{l:Qmatrix}(b).  
We also define  $w_*: \pReals^k\, \to \ \pReals$  by
\begin{equation}\label{def:mu-star}
	w_*(\xvec) := 
	\inf_{\calS \in \mathcal{F}(\xvec)} \  \sum_{ \yvec \in \calS}\ w(\yvec),       \qquad{\xvec \in \pReals^k},
\end{equation}
where 
$ \mathcal{F}(\xvec)$ 
consists of finite systems  $\calS$ of vectors from $\pReals^k$ such that $ \sum_{\yvec \in \calS}  \yvec = \xvec.$
In fact, the infimum of  $  \sum_{ \yvec \in \calS} w(\yvec)$ in \eqref{def:mu-star} is  always achieved by a system $\calS \in  \mathcal{F}(\xvec)$ consisting of  at most $k$ vectors; see Theorem \ref{l:Qmatrix}(e).

For a $k$-block graphon $W \equiv \sum_{i,j \in [k]}  p_{ij}  \one_{S_i \times S_j}$ and a vector  $\uvec \in [0,1]^n$, 
the model $\M(\uvec,W)$ is equivalent to  the stochastic block model $\M(\n,P)$, where 
\[
P=P(W) :=  (p_{ij})_{i,j \in [k]}
\] 
and  	$\n = \n(\uvec,W)  \in \Naturals^k$  is  defined by 
\begin{equation}\label{n-def}
	\n(\uvec,W) := (n_1(\uvec,W),\ldots,n_k(\uvec,W))^T\in \Naturals^k,
\end{equation}
where, for each  $i\in [k]$,		 
\[
n_i(\uvec,W) := \left|\{t \in [n] \st u_t \in S_i\}\right|.
\]
Recalling  that the components  of $\X(n) \in [0,1]^n$ are sampled  uniformly at random from $[0,1]$, we find that   
\[
\E \left[\n(\X(n),W) \right]=  n\cdot \uvec^W,
\]
where   
\begin{equation}\label{def:uW}
	\uvec^W  := (\lambda(S_1),\ldots, \lambda(S_k))^T.
\end{equation}
Next lemma establishes the relation between the functions $w(\cdot)$, $w_*(\cdot)$ defined in \eqref{w_def}, \eqref{def:mu-star}  and the quantities  $\varphi(W)$, $\varphi_*(W)$ defined in \eqref{def:var}, \eqref{def:var-star}.
\begin{lemma}\label{L:phi-relation}
	If  $W:[0,1]^2 \rightarrow [0,1)$ is a  block-graphon, then 
	\[
	\varphi(W) = w( \uvec^W) \qquad \text{and} \qquad \varphi_*(W) = w_*(\uvec^{W}).
	\]
\end{lemma}
\begin{proof}
	Since $W$ is a block graphon, we have $W \equiv \sum_{i,j \in [k]}  p_{ij}  \one_{S_i \times S_j}$
	for some  partition $(S_i)_{i\in[k]}$ of $[0,1]$. Let $Q = Q(W)$ be defined by \eqref{defQW}. For any measurable $S \subset [0,1]$, observe that
	\[
	\int_{S } \int_S\log \left(\dfrac{1}{1-W(x,y)}\right)  dx \,dy = \dfrac{\yvec^TQ \yvec}{ \|\yvec\|},
	\]
	where $\yvec = (y_1,\ldots,y_k)^T$ and $y_i = \lambda(S_i \cap S)$ for all $i \in [k]$. 
	Note that $\lambda(S_i \cap S) \leq \lambda(S_i)$.  Thus, 
	maximising the LHS over $S$ is equivalent to maximising of the RHS over $\yvec \preceq \uvec^W$. This establishes 	$\varphi(W) = w( \uvec^W)$. Similarly, the infimum in the definition of $\varphi_*(W)$
	corresponds to the infimum in the definition of $w_*(\yvec^W)$.
\end{proof}

Let
\begin{equation}\label{def:q-star}
	q^*:= \max_{i\in [k]}q_{ii} \qquad \text{and} \qquad \hat q(\xvec):= \begin{cases}\displaystyle \frac{\sum_{i\in [k]} x_i q_{ii}}
		{\|\xvec\|}, & \text{if $\xvec \neq \boldsymbol{0}$,}
		\\
		q^*, & \text{otherwise.}
	\end{cases}
\end{equation}

In order to prove Theorem~\ref{T:sequences1} we need the following result \cite[Theorem 2.1]{IsaevKang-SBM} on the chromatic number in the stochastic block model -- in fact we use its consequence Corollary~\ref{C:cond1}.  

\begin{thm}\label{Thm_sbm}
	Let $\sigma \in [0, \sigma_0]$  for some fixed  $\sigma_0 < \frac14$
	and 
	let   $P = (p_{ij})_{i,j\in [k]}$ be such that $p_{ij}=p_{ji} \in [0,1)$  for all $i,j \in [k]$.
	Assume that  the following asymptotics hold:
	\begin{equation}\label{ass_sbm}
		\|\n\| \rightarrow \infty, \qquad k= \|\n\|^{o(1)}, \qquad  q^* =    \|\n\|^{-\sigma+o(1)}, \qquad \hat{q}(\n) =    \|\n\|^{-\sigma+o(1)}. 
	\end{equation}
	Assume also that  
	\begin{equation}\label{ass_sbm2}
		\left(1+ \dfrac{1}{q^*}\right)\max_{i,j \in [k]} q_{ij} \ll \log\|\n\| 
		\qquad 
		\text{ and }
		\qquad
		w_*(\n) \gg \dfrac{ k \hat q(\n) q^* \|\n\|    }{\log \|\n\|}.
	\end{equation}
	Then, whp 
	\begin{equation*}\label{eq_sbm}
		\chi(\G) = 
		(1+o(1))\dfrac{w_*(\n)}{ 2(1-\sigma)\log \|\n\| }, \quad \text{ where  $\G \sim \SBM$.}
	\end{equation*}
\end{thm}

We also state from  \cite[Theorem 2.6]{IsaevKang-SBM}  the following  useful facts about the the functions $w(\cdot)$
and  $ w_*(\cdot) $ that will be needed in the proofs later.

\begin{thm}\label{l:Qmatrix}
	Let $Q = (q_{ij})_{i,j \in [k]}$ be a  symmetric $k\times k$ matrix with non-negative entries.  
	Let $q^*$ and $\hat{q}(\cdot)$ be defined according \eqref{def:q-star}.
	Then, the following hold for any $\x=(x_1,\ldots,x_k)^T\in \pReals^k$.
	\begin{itemize}
		\item[\rm (a)] {\rm[Scaling and monotonicity].}	            If $\xvec' \in \pReals^k$ and $ \xvec'   \preceq s\xvec$ for some $s>0$, then
		$w(\xvec') \leq s w(\xvec)$ and $w_*(\xvec') \leq  s w_*(\xvec)$. 
		In particular, $w(s\xvec) = sw(\xvec)$ and   $w_*(s\xvec) = s w_*(\xvec)$. 
		\item[\rm(b)] {\rm [Corner maximiser].} There is $\zvec = (z_1,\ldots, z_k)^T$ with $z_i \in \{0,x_i\}$ for all $i\in [k]$
		such that 
		\[\dfrac{\zvec^T \, Q\,  \zvec}{\|\zvec\|} = w(\xvec):=\max_{\boldsymbol{0} \preceq \yvec \preceq \xvec}\ \dfrac{\yvec^T\, Q\, \yvec}{ \|\yvec\|}.\]

		\item[\rm(c)] {\rm[Pseudodefinite property].}  If $\yvec^T Q\yvec \geq 0$  for all $\yvec \in \Reals^k$ with $\sum_{i\in [k]}y_i =0$, then
		\[
		w(\xvec) = w_*(\xvec):=  \inf_{\calS \in \mathcal{F}(\xvec)} \  \sum_{ \yvec \in \calS}\ w(\yvec).
		\]

		\item[\rm(d)] {\rm[Upper and lower bounds].}  
		We have
		\[ 
		q^* \|\xvec\| \geq \hat{q}(\xvec) \|\xvec\|    \geq  w_*(\xvec) \geq  \dfrac{\left(\hat{q}(\xvec) \right)^2}{\sum_{i \in [k]}q_{ii}}\|\xvec\| \geq   \dfrac{\left(\hat{q}(\xvec) \right)^2}{kq^*}\|\xvec\|, 
		\]
		where  the lower bounds for $w_*(\xvec)$ hold  under the additional condition  that $q^*>0$.

		\item[\rm(e)] {\rm[Triangle inequality].}  For any $\xvec' \in \pReals^k$, we have  
		$
		w_*(\xvec) + w_*(\xvec') \geq w_*(\xvec+\xvec').  
		$ 
		

		\item[\rm(f)] {\rm[Minimal system of $k$ vectors].}  There exists a system of vectors $(\x^{(t)})_{t \in [k]}$,  each from  $ \pReals^k $,
		such that 
		$\sum_{t \in [k]} \x^{(t)} = \x$    	
		and  $\sum_{t \in [k]}  w(\xvec^{(t)}) = w_*(\xvec)$. 
		%
	\end{itemize}
\end{thm}

Recalling that  $\calG(\uvec,W)$ is equivalent to the stochastic block model for a block graphon $W$, we get the following corollary.

\begin{cor}\label{C:cond1}
	Let  $\eps,\eps'>0$ be fixed  and   $k = k(n) \in \Naturals$ satisfy
	$k =  o (\log n)$.	
	Suppose 
	$k$-block graphons   $W= W(n):  [0,1]^2 \rightarrow [0,1]$ 
	and   vectors  $\uvec= \uvec(n) = (u_1 \ldots u_n)^T \in [0,1]^n$ satisfy  
	\[
	\sup_{x,y \in [0,1]} W(x,y)\leq 1-\eps \qquad 
	\text{and}
	\qquad
	w_*(\n(\uvec,W))  \geq \eps' n,
	\]
	where $Q = Q(W)$ is the matrix from \eqref{defQW} and  $w_*(\cdot)$ is   defined  by \eqref{def:mu-star}. 
	Then,   whp 
	\[
	\chi(\G^{\uvec}) =   
	(1+o(1)) \dfrac{w_* (\n(\uvec,W))}{2\log n},    \qquad \text{where $\G^{\uvec} \sim \calG(\uvec,W)$.}
	\]	
\end{cor}
\begin{proof} 
	We   check the assumptions of Theorem \ref{Thm_sbm}   with 
	$\n = \n(\uvec, W)$, $P=P(W)$, and  $\sigma = 0$.  Note that $\|\n(\uvec,W)\| = n$.  Observe that 
	\[
	k = o(\log n) = n^{o(1)} \qquad \text{and}	\qquad  q^* \leq  \max_{x,y\in [0,1]} \log \left(\dfrac{1}{1- W(x,y)}\right) \leq \log \dfrac{1}{\eps} = n^{o(1)}. 
	\]
	From the upper bound of  Theorem~\ref{l:Qmatrix}(d) and the assumptions of corollary, we have  that
	\[
	q^* \geq \hat q(\n(\uvec,W))  \geq   \dfrac{1}{n} w_*( \n(\uvec,W))  \geq \eps'   = n^{o(1)}.  
	\]
	This verifies \eqref{ass_sbm}. Furthermore,  we find that
	\[
	\left(1+ \dfrac{1}{q^*}\right) \max_{x,y\in [0,1]} \log \left(\dfrac{1}{1- W(x,y)}\right)  = O(1) \ll \log n
	\] 	
	and 
	\[
	\frac{ n k q^*  \hat q (\n(\uvec,W)) } {w_*(\n(\uvec,W))} = O( k ) \ll \log n.
	\]
	This verifies \eqref{ass_sbm2}. Applying Theorem \ref{Thm_sbm} completes the proof.
\end{proof}

%



\subsection{Proof of  Theorem \ref{T:sequences1}}
\label{S:sequences} 
For notation simplicity, we suppress the subscript and let  $ W = W_n \equiv \sum_{i,j\in [k]} p_{ij} \boldsymbol{1}_{S_i \times S_j}$
throughout this section. 

First, we show that  it  is sufficient to prove Theorem \ref{T:sequences1} under additional assumption that 
$\varphi_*(W) = \Omega(1)$. Indeed, if $\varphi_*(W) = o(1)$ then we consider a perturbed $k$-block graphon $W'$ defined by
\[
W'\equiv \sum_{i,j\in [k]} p_{ij}' \boldsymbol{1}_{S_i \times S_j}, \qquad  \text{where }
p_{ij}' :=  
(1-\eps')p_{ij} + \eps'.
\]
Since $1-p_{ij}' = (1-\eps')(1-p_{ij})$, we get that 
\[
\log \left(\dfrac{1}{1-W'(x,y)}\right) =
\log \left(\dfrac{1}{1-\eps'}\right)+ \log \left(\dfrac{1}{1-W(x,y)}\right).
\]
Then, for any measure $\mu \in \calP([0,1])$, we have  
\begin{align*}
	\dfrac{1}{\mu(S)}\
	\int_{S}\int_{S}	&\log \left(\dfrac{1}{1-W'(x,y)}\right)  d\mu(x)d\mu(y)
	=   
	\\
	&=	\dfrac{1}{\mu(S)}\
	\int_{S}\int_{S}	 \log \left(\dfrac{1}{1-W(x,y)}\right)  d\mu(x)d\mu(y) 	
	+  \mu(S) \log \left(\dfrac{1}{1-\eps'}\right).
\end{align*}
Recalling definitions  \eqref{def:var}, \eqref{def:var-star},   we get
\[
\log \left(\dfrac{1}{1-\eps'}\right) \leq \varphi(\mu,W') \leq \varphi(\mu,W) + \log \left(\dfrac{1}{1-\eps'}\right).
\] 
Taking  arbitrary representations of the uniform measure $\lambda$ 
as a finite convex combination of probability measures $\mu$, we derive that
\[
\log \left(\dfrac{1}{1-\eps'}\right) \leq 	\varphi_*(W')\leq \varphi_*(W) + \log \left(\dfrac{1}{1-\eps'}\right).
\]
In particular, the perturbed graphon $W'$ satisfies the additional assumption 
$\varphi_*(W') = \Omega(1)$. If Theorem \ref{T:sequences1} is true for this case, we establish whp
\[
\chi(\G') \leq O(\eps') \dfrac{n}{\log n},\qquad \text{where $\G' \sim \calG(n,W').$}
\]
Since $W'(x,y) \geq W(x,y)$, there is a coupling $(\G,\G')$ such that 
$\G \sim \calG(n,W)$  and $\G \subset \G'$. Then, 
$
\chi(\G) \leq \chi(\G').
$
Taking $\eps'$ arbitrary small, we show $\chi(\G) = o(1) \dfrac{n}{\log n}$ whp, as required.
This
justifies our claim about the additional assumption $\varphi_*(W) = \Omega(1)$.

Now, we proceed to the proof of  Theorem \ref{T:sequences1} assuming 
$\varphi_*(W) = \Omega(1)$. 
Let $X_1,\ldots, X_n$ be   sampled uniformly independently  from $[0,1]$
and $\X = (X_1,\ldots,X_n)^T \in [0,1]^n$.
Let  $\uvec^W$ and   $\n(\X,W)$ be defined according to \eqref{def:uW} and \eqref{n-def}. 
If  we establish that whp
\begin{equation}\label{eq:w*-conc}
	w_*(\n(\X,W)) =  n \cdot w_*(\uvec^W) +o(n),
\end{equation}
then the result follows by Corollary \ref{C:cond1} and Lemma \ref{L:phi-relation}.

For \eqref{eq:w*-conc}, we first prove the following lemma.
Let   $ \boldsymbol{1} := (1,\ldots, 1)^T \in \Naturals^k$.
\begin{lemma}\label{L:con1}
	Let $k = k(n)$ and $\delta = \delta(n)$ be such that 
	$k = n^{O(1)}$  and $1 \gg \delta  \gg  \frac{1}{\log n}$.
	Then,  
	for any sequence   $W= W(n)$ of $k$-block graphons,  the inequalities 
	\[
	(1-  \delta)  n\cdot \uvec^{\mu} - \log^2 n\cdot \boldsymbol{1}  \preceq \n(\X,W)   \preceq (1+\delta)  n\cdot \uvec^{\mu} + \log^2 n\cdot \boldsymbol{1}
	\]
	hold with probability  $1 - e^{-\omega(\log n)}$.
\end{lemma}
\begin{proof}
	Recall from  \eqref{n-def} that   $\n(\X,W)  = (n_1, \ldots,n_k )^T$, where  
	$n_t   = \left|\{i \in [n] \st X_i  \in S_t\}\right|. 
	$
	For $t \in [k]$, observe that  $n_t  \sim \text{Bin}(n, \mu(S_t))$. 
	Applying the Chernoff bound (see, for example, \cite[Theorem 2.1]{JLR2000})
	with $\zeta = \delta \E  n_t + \log^2 n  $, we get that 
	\[
	\Pr(|n_t - \E n_t| > \zeta ) \leq 
	2 e^{-  \frac{\zeta^2}{ 2 (\E n_t + \zeta/3)} }
	=	e^{-\Omega(\zeta \delta)} = e^{-\omega(\log n)}.
	\]
	Applying the union bound for all coordinates $t\in [k]$, we complete the proof.
\end{proof}

Next,  combining the monotonicity and scaling properties and the triangle inequality of Theorem \ref{l:Qmatrix}(a,e) and  the upper bound of  Lemma \ref{L:con1} with any $1\gg \delta \gg \frac{1}{\log n}$, 
we get that, with probability 
$1 - e^{-\omega(\log n)}$, 
\begin{align*}
	w_*(\n (\X,W))
	\leq(1+ o(1)) n \cdot  w_*( \uvec^{\mu}) + \log^2n  \cdot w_*(\boldsymbol{1}).
\end{align*}
Using the assumption that $\sup_{x,y\in [0,1]}W(x,y) \leq 1-\eps$ and   the upper bounds of Theorem \ref{l:Qmatrix}(d), we find  that
\[
w_*( \uvec^{\mu})  = O(1)  \qquad \text{and} \qquad  w_*(\boldsymbol{1}) = O(k).
\]
Thus, we derive that
\[
w_*(\n (\X,W),Q) \leq n\cdot w_*(\uvec^{\mu})  +o(n).
\]

To establish the lower bound of \eqref{eq:w*-conc}, we use the other part of   Lemma \ref{L:con1} to conclude  that
\[
(1 +o(1))  n\cdot \uvec^{\mu}  \preceq \n(\X,W)   + \log^2 n\cdot \boldsymbol{1}.
\]
Combining the monotonicity and scaling properties and the triangle inequality of Theorem \ref{l:Qmatrix}(a,e), we get that
\[
w_*(\n (\X,W)) +    \log^2n  \cdot w_*(\boldsymbol{1}) 
\geq     (1+ o(1)) n \cdot  w_*( \uvec^{\mu}).
\]
The rest of the argument is similar to the proof of the upper bound. This completes the proof of \eqref{eq:w*-conc} and of Theorem \ref{T:sequences1}.

\subsection{Proof of Theorem \ref{thm:balanced} and Corollaries \ref{C:block1} and \ref{C:block2}} \label{S:balanced}
To prove  Theorem \ref{thm:balanced} and its corollaries, we need  the following two lemmas that hold for arbitrary $k\times k$ symmetric matrix $Q$ with non-negative entries.  Let   $w,w_*: \pReals^k  \to  \pReals$ be the functions defined by \eqref{w_def} and \eqref{def:mu-star} for such a matrix $Q$.
\begin{lemma}\label{l:identity}
	The following hold.
	\begin{itemize}
		\item[(a)] Let $\zvec, \zvec', \avec \in \pReals^k$. 
		If $\avec \preceq \zvec '\preceq \avec +\zvec$ and $\dfrac{\zvec'^TQ\zvec'}{\|\zvec'\|}\geq \dfrac{\zvec^TQ\zvec}{\|\zvec\|} \geq \dfrac{(\zvec+\avec)^TQ(\zvec+\avec)}{\|\zvec+\avec\|}$ then 
		\[\dfrac{\zvec'^TQ\zvec'}{\|\zvec'\|} \leq \dfrac{(\zvec'-\avec)^TQ(\zvec'-\avec)}{\|\zvec'\|}.\]
		\item[(b)] For any $\zvec,\zvec'\in \pReals^k$, we have
		\[
		\dfrac{\zvec^T \, Q\,  \zvec}{\|\zvec\|} + 
		\dfrac{\zvec'^T \, Q\,  \zvec'}{\|\zvec'\|}
		-        \dfrac{(\zvec+\zvec')^T \, Q\,  (\zvec+\zvec')}{\|\zvec+\zvec'\|}
		\\ = \dfrac{\|\zvec\|\|\zvec'\| }{\|\zvec\|+\|\zvec'\|} 
		\left(\dfrac{\zvec}{\|\zvec\|} - \dfrac{\zvec'}{\|\zvec'\|}\right)^T \, Q\, 
		\left(\dfrac{\zvec}{\|\zvec\|} - \dfrac{\zvec'}{\|\zvec'\|}\right).
		\]
	\end{itemize}
\end{lemma}
\begin{proof}
	Part (b) is given in \cite[Lemma 29]{IsaevKang-SBM}. For part (a),
	by assumptions, we have
	\[
	\dfrac{\zvec^T Q \zvec}{\|\zvec\|}   \geq\dfrac{(\zvec+\avec)^T Q (\zvec+\avec)}{\|\zvec+\avec\|} = 
	\dfrac{\zvec^T Q \zvec + \avec^T Q(2 \zvec + \avec) }
	{\|\zvec\| +\|\avec\|}.
	\]
	Since $a/c \geq (a+b)/(c+d)$ is equivalent $a/c \geq b/d$ for $a,b\geq 0$ and $c,d>0$, we derive that  
	\[
	\dfrac{\zvec^T Q \zvec}{\|\zvec\|} \geq \dfrac{\avec^T Q(2 \zvec + \avec) }{ \|\avec\|}. 
	\]
	Since $\zvec '\preceq \avec +\zvec$ and $Q$ has non-negative enties, we get that
	\[
	\dfrac{\zvec'^T Q \zvec'}{\|\zvec'\|} \geq  \dfrac{\avec^T Q(2 \zvec + \avec) }{ \|\avec\|} 
	\geq 
	\dfrac{\avec^T Q(2 \zvec' - \avec) }{ \|\avec\|}. 
	\]
	Then, 
	\[
	\dfrac{(\zvec'-\avec)^T Q (\zvec'-\avec)} {\|\zvec'-\avec\|}
	= \dfrac{\zvec'^T Q \zvec' - \avec^T Q(2 \zvec' - \avec) }{\|\zvec'\| - \|\avec\|}
	\geq \dfrac{\zvec'^T Q \zvec'}{\|\zvec'\|}
	\]
	as required. 
\end{proof}

Using Lemma \ref{l:identity}, we prove extended versions of the corner maximiser and pseudodefinite properties of Theorem \ref{l:Qmatrix}(b,c). We also establish a nesting property for corner maximisers. 

\begin{lemma}\label{L:extended} 
	Let  $\xvec =(x_1,\ldots,x_k)^T \in \pReals^k$.  Let 
	\[ 
	M_{\xvec}:= \{\zvec \in \pReals^k \st z_i\in\{0,x_i\} \text{ and }
	\dfrac{\zvec^TQ\zvec}{\|\zvec\|} = w(\xvec) \}.
	\]
	\begin{itemize}
		\item[(a)] {\rm[Nesting].}  If  $\zvec,\zvec'\in M_{\xvec}$   then 
		$
		\hat{\zvec} \in M_{\xvec}
		$
		where $\hat{\zvec}$ is defined by $\hat{z_i} := \min\{z_i,z_i'\}$ \  for all $i\in [k]$.

		\item[(b)] A vector  $\zvec \in M_{\xvec}$  if and only if 
		$z_i\in\{0,x_i\}$ for all $i \in [k]$,    
		$w(\zvec) = \dfrac{\zvec^TQ\zvec}{\|\zvec\|}$ and, for any $\zvec'\succeq \zvec$ with $z_i'\in\{0,x_i\}$,  we have 
		$\dfrac{\zvec'^TQ\zvec'}{\|\zvec'\|} \leq w(\zvec)$.

		\item[(c)] Let $\zvec$ be a vector from   $M_{\xvec}$ with the smallest $\|\zvec\|$. Then $w(\xvec) = w_*(\xvec)$ if and only if 
		$\yvec^T Q \yvec \geq 0$ for all $\yvec \in \Reals^k$ with $\sum_{i\in [k]}  y_i = 0$ and   $y_i = 0$ whenever $z_i =0$. 
	\end{itemize}
\end{lemma}
\begin{proof}
	For part (a), we use  Lemma \ref{l:identity}(a) with $\avec:= \zvec' - \hat{\zvec}$. Then, we get that 
	\[
	\dfrac{\hat{\zvec}^TQ\hat{\zvec}}{\|\hat{\zvec}\|}  \geq \dfrac{\zvec'^TQ\zvec'}{\|\hat{\zvec'}\|} = w(\xvec). 
	\]
	Recalling
	\[
	w(\xvec):= \max_{\boldsymbol{0}\preceq \yvec \preceq \xvec} \dfrac{\yvec^TQ\yvec}{\|\yvec\|},
	\]
	and observing $\hat{\zvec}\preceq  \xvec$, we prove  (a).
	
	One direction in part (b) is straightforward. Indeed,    if $\dfrac{\zvec^TQ\zvec}{\|\zvec\|} = w(\xvec)$ then 
	$\dfrac{\zvec^TQ\zvec}{\|\zvec'\|} \leq \dfrac{\zvec^TQ\zvec}{\|\zvec\|}$
	for all $\zvec' \preceq \xvec$.
	Using the monotonicity property of Theorem \ref{l:Qmatrix}(a), we also find that $w(\zvec) = w(\xvec)$.  
	For the other direction, by the corner maximiser property of Theorem \ref{l:Qmatrix}(b), there is  $\zvec' \in M_{\xvec}$. If $\zvec \preceq \zvec'$ or $\zvec' \preceq \zvec$ then, by assumptions and the definition of $w(\cdot)$, we get $\dfrac{\zvec'^TQ\zvec'}{\|\zvec'\|}  \leq w(\zvec)$, which implies $w(\zvec) = w(\xvec)$. Otherwise, consider $\avec:=\zvec'-\hat{\zvec}$, where $\hat{\zvec}$ is defined as in part (a). Applying Lemma \ref{l:identity}(a), we get that 
	\[
	\dfrac{(\zvec'-\avec)^T Q (\zvec'-\avec)} {\|\zvec'-\avec\|}
	\geq \dfrac{\zvec'^TQ\zvec'}{\|\zvec'\|} = w(\xvec).
	\]
	Therefore, $\zvec'-\avec \in M_{\xvec}$. Observing that $\zvec'-\avec \preceq \zvec$, we get $w(\zvec) = w(\xvec)$ as required.

	For part (c), we take $\zvec$ to be a vector with smallest non-zero $\|\zvec\|$
	such that $z_i \in \{0,x_i\}$  
	and  $w(\xvec) = \dfrac{\zvec^T \, Q\,  \zvec}{\|\zvec\|}$.
	Let $\hat{Q}$ be the matrix obtained from $Q$ by replacing with zeros all   columns and rows  corresponding to zero components of $\zvec$. Let $\hat{w}$  and $\hat{w}_*$ be the functions defined by \eqref{w_def} and \eqref{def:mu-star} for matrix $\hat{Q}$. Note that 
	\[
	w(\xvec) = w(\zvec) = \hat{w}(\zvec)=\hat{w}(\xvec) \qquad \text{and} \qquad w_*(\xvec) \geq \hat{w}_*(\xvec).
	\]
	If   $\yvec^T Q \yvec \geq 0$ for all $\yvec \in \Reals^k$ with $\sum_{i\in [k]}  y_i = 0$ and   $y_i = 0$ whenever $z_i =0$ then 
	$\yvec^T \hat{Q} \yvec \geq 0$ for all $\yvec \in \Reals^k$ with $\sum_{i\in [k]}  y_i = 0$. Applying the pseudodefinite property of Theorem \ref{l:Qmatrix}(c), we get that    $\hat{w}_*(\xvec) = \hat{w}(\xvec)$. Thus,
	\[
	\hat{w}_*(\xvec) = \hat{w}(\xvec)  = w(\xvec) \geq w_*(\xvec) \geq \hat{w}_*(\xvec), 
	\]
	which implies that  $ w(\xvec) = w_*(\xvec)$ as required.
	
	The proof of  the other direction of part (c)  is  by contradiction to our assumption that $w_*(\xvec)=w(\xvec)$. 
	Assume that there is $\yvec$  with $\sum_{i\in [k]}  y_i = 0$ and   $y_i = 0$ whenever $z_i =0$  such that 
	$\yvec^T Q \yvec < 0$. 
	Observe that $\zvec \preceq \zvec'$ for any $\zvec' \in M_{\xvec}$.
	We have  $\xvec = \frac12 \xvec^{+} +\frac12 \xvec^-$, where $\xvec^{\pm} =   \xvec \pm \eps \yvec$. Using the corner maxmiser property of Theorem \ref{l:Qmatrix}(b) for $\xvec^+$ and $\xvec^-$ and taking $\eps$ sufficiently small, we 
	can find some $\zvec^+, \zvec^- \in M_{\xvec}$ such that
	\[
	w(\xvec^+) = \dfrac{(\zvec^++\eps \yvec)^T Q (\zvec^++\eps \yvec)}{\|\zvec^+ +\eps \yvec\|}, \qquad 
	w(\xvec^-) = \dfrac{(\zvec^- -\eps \yvec)^T Q (\zvec^--\eps \yvec)}{\|\zvec^--\eps \yvec\|}.
	\]
	By the choice of $\yvec$,  note that $\|\zvec^+ +\eps \yvec\| = \|\zvec^+\|$
	and  $\|\zvec^- -\eps \yvec\| = \|\zvec^-\|$.  Applying Lemma~\ref{l:identity}(b), we get that 
	\begin{align*}
		\dfrac{(\zvec^++\eps \yvec)^T Q (\zvec^++\eps \yvec)}{\|\zvec^+ +\eps \yvec\|} + \dfrac{(\zvec^--\eps \yvec)^T Q (\zvec^--\eps \yvec)}{\|\zvec^--\eps \yvec\|} - \dfrac{(\zvec^++\zvec^-)^T Q (\zvec^++\zvec^-)}{\|\zvec^++\zvec^-\|} =  \dfrac{\|\zvec^+\| +\|\zvec^-\|}{
			\|\zvec^+\|\|\zvec^-\|
		}\eps^2 \yvec^T Q\yvec. 
	\end{align*}
	Since $\yvec^TQ \yvec<0$ and $\zvec^+ +\zvec^- \preceq 2\xvec$, we get that
	\[
	w( \xvec^+) + w(\xvec^-)
	< \dfrac{(\zvec^++\zvec^-)^T Q (\zvec^++\zvec^-)}{ \|\zvec^++\zvec^-\|}
	\leq w(2\xvec).
	\]
	Using the scaling property of Theorem \ref{l:Qmatrix}(a), we conclude that 
	\[
	w(\tfrac12  \xvec^+) + w(\tfrac12 \xvec^-) < w(\xvec).
	\]
	This contradicts to the assumption $w_*(\xvec) = w(\xvec)$, recalling the definition of $w_*(\cdot)$ from \eqref{def:mu-star}.
\end{proof}

Now we are ready to prove Theorem \ref{thm:balanced} and its corollaries.

\begin{proof}[Proof of Theorem \ref{thm:balanced}.]
	From Lemma \ref{L:phi-relation}, we know that $\varphi_*(W) = \varphi(W)$ 
	is equivalent to $w_*(\uvec^W) = w(\uvec^W)$, where $\uvec^W$ is defined by \eqref{def:uW}. We complete the proof by applying   Lemma \ref{L:extended}(c) for $\xvec= \uvec^W$.  
\end{proof}

\begin{proof}[Proof of Corollary \ref{C:block1}]
	In this case, we have  
	\[
	Q(W) = I \,\log \dfrac{1}{1-p_0} + (J-I) \log \dfrac{1}{1-p}
	= I\,\log \dfrac{1-p}{1-p_0} + J \, \log \dfrac{1}{1-p},
	\] where $I$ is the indentity matrix and $J$ is  the matrix with all entries equal $1$. Since $p\leq p_0$, we get that $Q(W)$ is a positive semidefinite matrix.
	Combining Theorems \ref{T:block} and \ref{thm:balanced}, we get that whp
	\[
	\chi(\G) = (\varphi(W)+o(1))\dfrac{n}{2\log n}. 
	\]
	Using the corner maximiser property of Lemma \ref{l:Qmatrix}(b), we obtain
	\[
	\varphi(W) = w(\uvec^W) = \max_{U \subseteq [k]}  R(U),
	\]
	where 
	\[
	R(U) := \dfrac{\sum_{j \in U} \ell_j^2}{\sum_{j\in U} \ell_j}  \log \dfrac{1-p}{1-p_0}
	+\sum_{j\in U} \ell_j\log \dfrac{1}{1-p}.
	\]
	It remains to show that the maximum of $R(U)$ occurs at $U = [i]$ for some $i\in [k]$. 
	
	Indeed, suppose there is  $a\in [k]\setminus U$ and $b >a$ such that $b \in U$. 
	Let
	\[ 
	S_1 := \sum_{j \in U\setminus\{b\}} \ell_j, \qquad
	S_2 := \sum_{j \in U\setminus\{b\}} \ell_j^2.
	\]
	Consider the function $f:\pReals \rightarrow \pReals$ defined by
	\[
	f(x):=\dfrac{x^2 + S_2}{x +S_1}  \log \dfrac{1-p}{1-p_0}
	+(x+S_1) \log \dfrac{1}{1-p}.
	\]
	A direct computation shows that 
	\[
	f''(x) = \tfrac{2(S_2+S_1^2)}{(x+S_1)^2} \log \dfrac{1-p}{1-p_0}\geq 0.
	\]
	That is, $f(x)$ is convex. 
	Since $0\leq  \ell_b \leq \ell_a$,  we have 
	$f(\ell_b) \leq \max\{f(0), f(\ell_a)\}$. This implies that, we can either remove $b$ from $U$ or replace it by $a$ without decreasing $R(U)$. We can do it until $U$ becomes $[i]$ for some $i \in [k]$.  This completes the proof.
\end{proof}

\begin{proof}[Proof of Corollary \ref{C:block2}]
	In this case, we have  
	\[
	Q(W) =  \operatorname{diag}\left( \(\log \dfrac{1}{1-p_i}\)_{i\in [k]}\right)+ (J-I) \log \dfrac{1}{1-p}
	=  \operatorname{diag}\left( \(\log \dfrac{1-p}{1-p_i}\)_{i\in [k]}\right) + J \, \log \dfrac{1}{1-p},
	\] where  
	$\operatorname{diag}(\uvec)$  denotes the diagonal matrix with diagonal entries equal to components of $\uvec$. Since $p$ is smaller than any of $p_i$, we get that $Q(W)$ is a positive semidefinite matrix.
	Combining Theorems \ref{T:block} and \ref{thm:balanced}, we get that whp
	\[
	\chi(\G) = (\varphi(W)+o(1))\dfrac{n}{2\log n}. 
	\]
	Using the corner maximiser property of Lemma \ref{l:Qmatrix}(b), we obtain
	\[
	\varphi(W) = w(\uvec^W) = \max_{U \subseteq [k]} \left( \dfrac{1}{k|U|}  \sum_{j\in U}\log \dfrac{1}{1-p_j}
	+\dfrac{|U|-1}{k}\log \dfrac{1}{1-p}\right).
	\]
	Clearly, among all sets $U$ of size $i$ the maximum of the expression above occurs when  the  numbers  $p_j$ for $j\in U$ are as big as possible. This occurs when $U=[i]$, that is,   we take $p_1,\dots,p_i$. The claimed formula follows.
\end{proof}


\subsection{Chromatic number of the graphon from Figure \ref{f:block}}\label{S:example}

From Theorem \ref{T:block} and Lemma \ref{L:phi-relation}, we know that  whp
\begin{equation}\label{ex1}
	\chi(\G) = (1+o(1))w_*\left(\dfrac13 \cdot \one\right) \dfrac{n}{2\log n}, 
\end{equation}
where $\G \sim \M(n,W)$, $\one = (1,1,1)^T$, and $w_*$ is defined by \eqref{def:mu-star}  for the matrix 
\[
	Q = Q(W) = \log 2 \cdot
	\begin{pmatrix}
			1 &1& 2\\
			1 &2 &1\\
			2& 1 &3
		\end{pmatrix}.
\]
Using the scaling property and the existence of minimal system of  $k$ vectors from Theorem~\ref{l:Qmatrix}(a,f), we get that
\begin{equation}\label{ex2}
	w_*\left(\dfrac13 \cdot \one\right)  = \dfrac13 \left(w(\xvec^{(1)}) + w(\xvec^{(2)}) + w(\xvec^{(3)})\right)
\end{equation}
for some $\xvec^{(1)},\xvec^{(2)},\xvec^{(3)} \in \pReals^3$ and 
$\xvec^{(1)} +\xvec^{(2)}+\xvec^{(3)} = \one$. 
 Note that, for each $i  \in [3]$,
\[
		w_*(\xvec^{(i)}) = w(\xvec^{(i)}),
\]
since otherwise, we could have split $\xvec^{(i)}$ and reduce the value of 
$	w_*\left(\dfrac13 \cdot \one\right) $. Next step is to classify all $\xvec$ such that 
$w_*(\xvec)=w(\xvec)$ using Lemma \ref{L:extended}. 
\begin{lemma}\label{l:x123}
	 We have	$w_*(\xvec)=w(\xvec) = \log 2$ if and only if $\xvec = (x_1,x_2,x_3)^T \in \pReals^3$ lies  in the union of the following curves:
	   \begin{itemize}
	   	    \item[(i)] $x_3 = \frac13$, $x_1 = 0$, $x_2 \leq \frac16$;
	   	    	\item[(ii)] $2x_2^2 + 2x_2 x_3 + 3x_3 ^2 = x_2 + x_3$, $x_1 =0$, and $x_2 \geq \frac16$;
	   		\item[(iii)]  $x_1 +3x_3 =1$ and $x_2 =0$;
	   		\item [(iv)] $x_1^2  +2x_1 x_2 + 2x_2^2 =x_1 +x_2$, $x_2>0$, and $x_3=0$.
	   	\end{itemize}
	\end{lemma}
\begin{proof}
	First, we show that any point $\xvec$ satisfying at least one of conditions (i)--(iv) is indeed such that $w_*(\xvec)=w(\xvec) = \log 2$. Using Lemma \ref{L:extended}(b), it is straightforward to check the following: 
	\begin{itemize}
			\item[-] If $\xvec$ satisfies (i) then $(0,0,\frac13)^T \in M_{\xvec}$;
			\item[-] If $\xvec$ satisfies (ii) then $(0,x_2,x_3)^T \in M_{\xvec}$;
			\item[-] If $\xvec$ satisfies (iii) then $(x_1,0,x_3)^T \in M_{\xvec}$;
			\item[-] If $\xvec$ satisfies (iv) then $(x_1,x_2,0)^T \in M_{\xvec}$.
 		\end{itemize} 
 	Also we get that $w(\xvec) = \log 2$ in all cases.
	Observe  that all principal submatrices  $\hat{Q}$ of $Q$ except $Q$ itself satisfy the pseudodefinite property 
$
	\yvec^T \hat{Q} \yvec \geq 0
$
	for all $\yvec$ with zero sum of components. Then, using  Lemma \ref{L:extended}(c), we justify that $w_*(\xvec) = w(\xvec)$.
	
	Next, assume that $w_*(\xvec) = w(\xvec)=\log 2$. From Lemma \ref{L:extended}(c), we know that  there is $\zvec\in \pReals^3$ with 
	$z_i \in \{0,x_i\}$ such that 
	$
			\dfrac{\zvec^TQ\zvec}{\|\zvec\|} =  \log 2
	$
	and non-zero components  correspond to a principal submatrix  $\hat{Q}$ of $Q$ satisfying the pseudodefinite property.  Note that $\zvec$ has at least one zero component $(2,-1,-1) Q (2,-1,-1)^T =   -2 \log 2 <0$. It remains to consider six possible cases specifying non-zero components of $\zvec$ (at least one component must be zero and at least one component must be non-zero). Four of them immediately give 
	(i)--(iv). It remains to check $\zvec = (x_1,0,0)^T$ and $\zvec = (0,x_2,0)^T$. 
	 If $\zvec = (x_1,0,0)^T$ then   
	\[
		x_1 = 1, \qquad \dfrac{x_1^2 +2x_1 x_2 + x_2^2}{x_1 +x_2} \leq 1, \qquad 
		\dfrac{x_1^2 +4x_1 x_3 + 3x_3^2}{x_1 +x_3} \leq 1.
	\]
	These conditions  imply that $\xvec = (1,0,0)^T$. This point  is already covered by (iii). Similarly,  If $\zvec = (0,x_2,0)^T$ then 
	\[
		x_2 = \dfrac12, \qquad \dfrac{x_1^2 +2x_1 x_2 + x_2^2}{x_1 +x_2} \leq 1, \qquad 
		\dfrac{2x_2^2 +2x_1 x_3 + 3x_3^2}{x_2 +x_3} \leq 1.
	\]
	These conditions  imply that $\xvec = (0,\frac12,0)^T$. This point  is already covered by (iv).
	\end{proof}

Note that Lemma \ref{l:x123} classifies all $\xvec$ such that $w_*(\xvec) = w(\xvec)$ because we can always scale $\xvec$ so $w(\xvec) = \log 2$.
 It is easy to compute that if $\xvec$ has $x_2 =0$ then 
\begin{equation}\label{eq:split}
w(\xvec) =  x_1 \log 2 + 3x_3 \log 2  = w( (x_1,0,0))^T +  w((0,0,x_3)^T).
\end{equation}
Also, if $x_1 =x_1' =0$ or $x_3 =x_3'=0$
then, using the pseudodefinite property of Theorem  \ref{l:Qmatrix}(c) for a $2\times 2$ principal submatrix $\hat{Q}$ of $Q$ corresponding to non-zero components of $\xvec$, we find that 
\begin{equation}\label{eq:combine}
w(\xvec) + w(\xvec') \geq w(\xvec + \xvec').
\end{equation}

Starting from vectors $\xvec^{(1)},\xvec^{(2)},\xvec^{(3)}$ in \eqref{ex2}, if any of them has both non-zero first and third component, then, by Lemma \ref{l:x123} it must have zero second component so
we can split it using \eqref{eq:split}. Then, using \eqref{eq:combine}  we can combine all vectors with the same support. Thus, we get that 
\begin{align*}
w_*\left(\dfrac13 \cdot \one\right)  &= \dfrac{1}{3} \min_{r \in [0,1]}\left(  w((1,r,0)^T) + w((0,1-r,1)^T) \right)
\\ 
&= \dfrac{\log 2}{3} \min_{r \in [0,1]} \left(\dfrac{1 + 2r + 2r^2 }{1+r} + \max\left\{3, \dfrac{3 + 2 (1-r)+ 2(1-r)^2 }{2-r}\right\}\right).
\end{align*}
First, we consider the case $r \geq \frac12$, for which 
$
	   \dfrac{3 + 2 (1-r)+ 2(1-r)^2 }{2-r} \leq 3 .
$
Then, we have that  
\[
	\dfrac{1 + 2r + 2r^2 }{1+r} + \max\left\{3, \dfrac{3 + 2 (1-r)+ 2(1-r)^2 }{2-r}\right\} 
	=\dfrac{1 + 2r + 2r^2 }{1+r} + 3 
	\geq \dfrac{14}{3}.
\]
For $r \leq \frac12$, observe that 
\[
\dfrac{1 + 2r + 2r^2 }{1+r} + \max\left\{3, \dfrac{3 + 2 (1-r)+ 2(1-r)^2 }{2-r}\right\}
= 2 + \dfrac{1}{1+r} + \dfrac{3}{2-r}.
\]
The minimum of this function is achieved when $3 (1+r)^2 =  (2-r)^2$ which gives 
$r = \frac{3\sqrt 3 -5}{2}<1/2$.  At this point, we get
\[
2 + \dfrac{1}{1+r} + \dfrac{3}{2-r} = \dfrac{10+2\sqrt 3}{3} < \dfrac{14}{3}.
\]
Recalling \eqref{ex1}, we conclude that  whp
\[
 \chi(\G) = 	(1+o(1))w_*\left(\dfrac13 \cdot \one\right) \dfrac{n}{2\log n}
 =   \left(\dfrac{(5+\sqrt 3)\log 2}{9} + o(1)\right) \dfrac{n}{\log n}.
\]

\section{From block graphons to general graphons}\label{S:general}
In this section we prove Theorems \ref{T:New1}, \ref{C:main2},  \ref{Thm:con-inc}, and \ref{T:block-regulated}.
Before proceeding to the proofs,  we first study in  Section \ref{S:stability}  how much  the quantities  $\varphi(\mu, W)$, $\varphi_k(W)$, $\varphi_*(W)$ defined in Section~\ref{S:dense} 
can change with respect to perturbations of the graphon $W$ in $\calL^\infty$-norm and $\calL^2$-norm.

\subsection{Stability estimates for $\varphi(\mu, W)$, $\varphi_k(W)$,  and $\varphi_*(W)$ }\label{S:stability}


For a measurable function $F: [0,1]^2 \rightarrow [-1,1]$ with respect to $\mu \in \calP([0,1])$, let 
\begin{align*}
	\|F\|_{\calL^{2}(\mu\times \mu)} := \left(\int_{[0,1]^2} |F(x,y)|^2 d\mu(x)\, d\mu(y)\right)^{1/2}.
\end{align*}  
%

\begin{thm}\label{T:norm} 
	Let $W$ and $W'$ be graphons such that
	\[
	\sup_{x,y \in [0,1]}W (x,y) \leq 1-\eps \quad \text{ and } \quad \sup_{x,y \in [0,1]}W' (x,y) \leq 1-\eps
	\]
	for some  $\eps\in (0,1)$.  For $\delta \in [0,1|$, let   
	$
		S_{\delta} \subseteq [0,1]
	$
	be such that 
	\[
			\sup_{x,y \notin S_\delta}|W (x,y) - W'(x,y)| \leq \delta.
	\]
	Then, the following hold.
	\begin{itemize}
		\item[(a)]  For any $\mu \in \calP([0,1])$,  we have
		\[
		\left|\varphi(\mu,W) - \varphi(\mu,W')\right| \leq \eps^{-1}   \min\left\{ \delta + 2\mu(S_{\delta}),  \|W-W'\|_{\calL^2(\mu \times \mu)}\right\}.
		\]
		\item[(b)] For any $k \in \Naturals$, we have
		\[
		|\varphi_k(W) - \varphi_k(W')| \leq \eps^{-1}\sqrt{k}\, \|W-W'\|_{\calL^2(\lambda \times \lambda)}.
		\]
		\item[(c)] For any $k \in \Naturals$, we have
		\[ |\varphi_k(W) - \varphi_k(W')|\leq \eps^{-1} \left(\delta+ 2\lambda(S_{\delta})\right),
		\]
			and
		\[ 
		  |\varphi_*(W) - \varphi_*(W')| \leq \eps^{-1} 
		\left(\delta+ 2\lambda(S_{\delta})\right).
		\]
	\end{itemize}
\end{thm}
\begin{proof}
	%
	
	For part (a), without loss of the generality, we may assume that  $\varphi(\mu, W) \geq \varphi(\mu, W')$.  By definition \eqref{def:var} and properties of supremum, we get 
	\begin{align*}
		\varphi(\mu,W) - \varphi(\mu,W') 
		&=
		\sup_{S\subseteq [0,1]: \mu(S)>0}  
		\dfrac{1}{\mu(S)} \int_{S\times S}   \log \left(\dfrac{1}{1-W(x,y)}\right)d\mu(x) d\mu(y)
		\\
		&\hspace{15mm}-	\sup_{S'\subseteq [0,1]: \mu(S')>0}  
		\dfrac{1}{\mu(S')} \int_{S'\times S'}   \log \left(\dfrac{1}{1-W'(x,y)}\right)d\mu(x) d\mu(y)  
		\\
		&\leq    
		\sup_{S\subseteq [0,1]: \mu(S)>0}  
		\dfrac{1}{\mu(S)} \int_{S\times S}   \left|\log \left(\dfrac{1}{1-W(x,y)}\right)  - \log \left(\dfrac{1}{1-W'(x,y)}\right)\right|   d\mu(x) d\mu(y). 
	\end{align*}
	Observe  by the Mean Value theorem that, for all $x,y\in [0,1]$, 
	\begin{equation*} 
		\left|\log \left(\dfrac{1}{1-W(x,y)}\right)  - \log \left(\dfrac{1}{1-W'(x,y)}\right)\right| \leq \eps^{-1} |W(x,y) - W'(x,y)|.  	   
	\end{equation*}	
Bounding $|W(x,y) - W'(x,y)| \leq \delta$ for $x,y \in S\setminus S_{\delta}$
and  $|W(x,y) - W'(x,y)| \leq 1$ otherwise, we get that
\[
	\int_{S\times S}  \left|W(x,y)  - W'(x,y)\right|d\mu(x) d\mu(y)  
	\leq \delta \mu^2(S)   +  2 \mu(S\cap S_{\delta}) \mu(S)
	\leq (\delta + 2 \mu(S_\delta)) \mu(S).
\]
Alternatively,	using  the Cauchy-Schwartz inequality, we estimate 
	\begin{align*}
		\int_{S\times S}  \left|W(x,y)  - W'(x,y)\right|d\mu(x) d\mu(y)  &\leq
		\\\left(\int_{S\times S}  1 \, d\mu(x) d\mu(y)\right)^{1/2} &\left(\int_{S\times S}  \left|W(x,y)  - W'(x,y)\right|^2 d\mu(x) d\mu(y)\right)^{1/2}
		\\ &\leq 
		\mu(S) \|W-W'\|_{\calL^2(\mu \times \mu)}.
	\end{align*}
	Combining the above inequalities, part (a) follows.

	For part (b), without loss of the generality, we may assume that  $\varphi_k(W) \geq \varphi_k(W')$.  Consider any representation $(\calM,\alphavec)$ 	of the uniform measure $\lambda$ as a finite  convex  combination $ \sum_{\mu \in \calM} \alpha_{\mu} \mu$ with $|\calM| = k$. Using  the inequality between the arithmetic mean and the quadratic mean (AM-QM), we get that 
	\begin{align*}
		\|W-W'\|_{\calL^2(\lambda \times \lambda)}^2  &= \int_{[0,1]^2} (W(x,y) - W'(x,y))^2 dx dy\\
		&=
		\sum_{\mu,\mu' \in \calM} \alpha_{\mu} \alpha_{\mu'} \int_{[0,1]^2}
		(W(x,y) - W'(x,y))^2 d \mu(x) d\mu'(y)
		\\
		&\geq \sum_{\mu\in \calM} \alpha_{\mu}^2  \|W-W'\|_{\calL^2(\mu \times \mu)}^2
		\stackrel{\text{AM-QM}}{\geq }
		k^{-1}\left(\sum_{\mu\in \calM} \alpha_{\mu}  \|W-W'\|_{\calL^2(\mu \times \mu)} \right)^2.
	\end{align*}
	Then, applying part (a), we find that
	\begin{align*}
		\varphi_k(W) &\leq \sum_{\mu \in M} \alpha_\mu \varphi(\mu,W)  \\&\leq \sum_{\mu \in M} \alpha_\mu \varphi(\mu,W')+ 
		\sum_{\mu \in M} \alpha_{\mu} \eps^{-1} \|W-W'\|_{\calL^2(\mu \times \mu)} 
		\\ &\leq\alpha_\mu \varphi(\mu,W')+   \eps^{-1}\sqrt{k} \|W-W'\|_{\calL^2(\lambda \times \lambda)}.
	\end{align*}
	Taking the infimum over all such representations $(\calM,\alphavec)$, we prove part (b).
	
	Finally, for part (c),  we may assume again that  $\varphi_k(W) \geq \varphi_k(W')$.
	Consider any representation $(\calM,\alphavec)$ 	of the uniform measure $\lambda$ as a  finite convex  combination $ \sum_{\mu \in \calM} \alpha_{\mu} \mu$
	with $|\calM|\leq k$.  
	 From part (a), we get that
	\begin{align*}
	|\varphi(\mu,W) - \varphi(\mu,W)| 
	\leq  \eps^{-1}  \left(\delta  + 2 \mu(S_{\delta})\right).
	\end{align*}
	Then, we get that 
	\begin{align*}
		\varphi_k(W) \leq 	\sum_{\mu \in \calM} \alpha_{\mu}\varphi (\mu,W)
		&\leq	\sum_{\mu \in \calM} \alpha_{\mu}\left(\varphi (\mu,W') + \eps^{-1}(\delta + 2\mu(S_\delta))\right)
		\\ &=
		\sum_{\mu \in \calM} \alpha_{\mu}\varphi(\mu,W') +  \eps^{-1}(\delta + 2\lambda(S_{\delta})).
	\end{align*}
	Taking the infimum over all representations $(\calM,\alphavec) \in \calC_{\lambda}$
	with $|\calM| \leq k$, we   complete the proof of the first bound. The second bound  follows by taking limit as $k \rightarrow \infty$.
\end{proof}


\subsection{Proof of Theorem \ref{T:New1}}\label{S:New1}

For $k \in \Naturals$, let
\[
W_k^+(x,y) := \sup_{x', y'} W(x,y),
\]
where the supremum is over $x', y' \in [0,1]$ such that $\lfloor kx' \rfloor = \lfloor kx \rfloor$.
That is,  $W_k^+$ is a $k$-block graphon, in which blocks correspond to the partition of $[0,1]$ into $k$ intervals of the same size and  values equal to the suprema of $W$ over blocks.

Let $\eps := 1- \sup_{x,y \in [0,1]} W(x,y).$
Take $k = k(n)$ to be slowly growing.  
Then, by assumptions, 
\begin{equation*}
	W
	\leq W_k^+ \leq 1-\eps 
	\qquad  \text{and} \qquad	\|W - W_k^+\|_{\calL^2(\lambda\times \lambda)} \rightarrow 0. 
\end{equation*}
Since $W \leq W_k^+$, 
we can find  a random graph  $\G_k^+$ such that
$\G_k^+ \sim \calG(n,  W_k^+)$ 
and $\G \subseteq \G_k^+$. 
Applying Theorem \ref{T:sequences1}, we get that  whp
\[
\chi(\G) \leq \chi(\G_k^+) = \left(\varphi_*(W_k) +o(1)\right) \dfrac{n}{2\log n}
\leq \left(\varphi(W_k^+) +o(1)\right) \dfrac{n}{2\log n}.
\]
Using Theorem \ref{T:norm}(a) with $\mu$ equal to the uniform measure, we  get that 
$\varphi(W_k^+) \rightarrow \varphi(W)$,
completing the proof.

\subsection{Proof of Theorem \ref{C:main2}}\label{S:main2}
Consider any representation $(\calM,\alphavec) \in \calC_{\lambda}$ 	of the uniform measure $\lambda$ as a  finite convex  combination $ \sum_{\mu \in \calM} \alpha_{\mu} \mu$.  
Recalling \eqref{disj_union} and 
using different colours for each $\G_{\mu}$, we have an upper bound  
\begin{equation} \label{chi_union}
	\chi(\G) \leq \sum_{\mu \in \calM}    \chi(\G_{\mu}).
\end{equation}
Next, we show that how Theorem \ref{T:New1} can be applied to estimate $\chi(\G_\mu)$.

We note  that 
\[
\alpha_{\mu} \mu(S) \leq   \sum_{\nu \in \calM} \alpha_{\nu} \nu(S) =  \lambda(S),
\]
therefore $\mu$ is  absolutely continuous with respect to $\lambda$.
Using the Radon-Nikodym theorem, there is a measurable  function  $f_{\mu}: [0,1] \rightarrow \pReals$ such that, for all measurable $S \subset [0,1]$,
\[
\mu(S) = \int_{S} f_{\mu} (x) dx.
\]
Note that $\sum_{\mu \in \calM} \alpha_{\mu} f_{\mu} \equiv 1$ and therefore
$f_{\mu} \in \calL^{\infty}[0,1]$.
Applying the probability integral transform $T_{\mu}$ for $\mu$ (the right inverse of the non-decreasing  distribution function $F_\mu(t):=\int_{x=0}^t f_{\mu}(x)dx$), we find  a transformation $T_{\mu}: [0,1]\rightarrow [0,1]$ such that 
\begin{equation}\label{Transform_eq}
	\lambda(T_{\mu}[S]) = \mu(S)
\end{equation}
for any measurable $S$. Define graphon $W_{\mu}$ by
\[
W_\mu(x,y) \equiv W(T_{\mu}[x], T_\mu[y]) \qquad \text{ for all $x,y\in [0,1]$}.
\]
By definition of $\G_{\mu}$, we find that
$
\G_{\mu} \sim \G(n_{\mu}, W_{\mu}), 
$
where   $n_\mu$ is  distributed   according to  $\operatorname{Bin}(n,\alpha_{\mu})$. Using Chernoff's bound, we get that  whp
\begin{equation}\label{nmu-con}
	n_{\mu} = (1+o(1))\alpha_{\mu} n
\end{equation}
for all $\mu \in \calM$. 
Applying Theorem \ref{T:New1}, we obtain that whp
\[	
\chi(\G_{\mu}) \leq \left(\varphi(W_{\mu}) + o(1)\right) \dfrac{n_{\mu}}{ 2 \log n_{\mu}}
=  \left(\varphi(W_{\mu}) + o(1)\right) \dfrac{ \alpha_{\mu}n }{ 2 \log n}.
\]

Finally, recalling definitions 
\eqref{def:var_uni} and \eqref{def:var}, we obtain from  \eqref{Transform_eq} that 
\[
\varphi(W_{\mu}) = \varphi(\mu,W).  
\]
Then, substituting the above bound for $\chi(\G_{\mu})$ into  \eqref{chi_union}, we get  whp
\[
\chi(\G) \leq \sum_{\mu \in \calM}\left(\varphi(W_{\mu}) + o(1)\right) \dfrac{ \alpha_{\mu}n }{ 2 \log n} =    \left(\sum_{\mu \in \calM} \alpha_{\mu} \varphi(\mu, W)  + o(1)
\right) \dfrac{n}{2\log n}.
\]


\subsection{Proof of Theorem \ref{T:block-regulated}}\label{S:b-regulated}

Let $\eps := 1- \sup_{x,y \in [0,1]} W(x,y).$ 
Since the graphon $W$ is  block-regulated,
for any $\delta>0$,  
there exists a block graphon $W^{\delta}$ such that 
\[	
\sup\limits_{x,y \in [0,1]} |W(x,y) - W^{\delta}(x,y)| \leq \delta.
\]
Then,  there exist two block graphons 
$W^{\delta}_L$ and $W^{\delta}_U$ such that
\begin{equation}\label{eq:sand1}
	W^{\delta}_L(x,y) \leq W(x,y) \leq W^{\delta}_U(x,y) \leq 1-\eps, \qquad \text{ for all } x,y\in [0,1],
\end{equation} 
and 
\begin{equation}\label{sand-diff}
	\esssup\limits_{x,y \in [0,1]} |W^{\delta}_U(x,y) - W^{\delta}_L(x,y)| \leq 2\delta.
\end{equation}
Indeed, one can take 
$W^{\delta}_L:= \max\{W^{\delta} - \delta, 0\}$ and $W^{\delta}_U:= \min\{W^{\delta} + \delta, 1-\eps\}$.
Using \eqref{sand-diff} and Theorem \ref{T:norm}(c), we get that 
\begin{equation}\label{sand-phi}
	\varphi_*(W) + 2 \delta/ \eps \geq 
	\varphi_{*}(W_{U}^{\delta}) \geq \varphi_{*}(W_{L}^{\delta})  \geq \varphi_*(W) -  2 \delta/ \eps.
\end{equation}

Using \eqref{eq:sand1}, we can find   random graphs $\G_L^{\delta}(n)$ and $\G_U^{\delta}(n)$
such that   $ \G_L^{\delta}(n) \sim	\M(n, W^{\delta}_L)$,   $\G_U^{\delta}(n)\sim	\M(n, W^{\delta}_U)$, and   	$\G_L^{\delta}(n)\subseteq  \G \subseteq \G_U^{\delta}(n)$.
Therefore,
\begin{equation*}
	\chi(\G_L^{\delta}(n)) \leq \chi(\G) \leq  \chi(\G_U^{\delta}(n)).
\end{equation*}
By \eqref{eq:sand1}, we get that  
$W_{L}^{\delta}$ and  $W_{U}^{\delta}$ 
satisfy the assumptions of   
Theorem~\ref{T:sequences1}. Thus,  whp as $n \rightarrow \infty$,
\begin{equation*} 
	\begin{aligned}
		\chi(\G_L^{\delta}(n)) =  \left(\varphi_*(W_L^{\delta}) + o(1)\right) \dfrac{n}{2\log n},\\
		\chi(\G_U^{\delta}(n)) =  \left(\varphi_*(W_U^{\delta}) + o(1)\right) \dfrac{n}{2\log n}.  
	\end{aligned}
\end{equation*}
Since we can take $\delta$ arbitrarily small, combining \eqref{sand-phi} with the sandwich arguments given above completes the proof.

\subsection{Proof of Theorem \ref{Thm:con-inc}}\label{S:b-cont}

For simplicity, we only consider for the case when $W$ is increasing or Lipschitz. The proof generalises easily for block-increasing and block-Lipschitz  graphons  by using the same argument within each interval forming blocks.

Similarly to the proof of Theorem \ref{T:New1}, for a positive integer $k$,
consider $k$-block graphons 
$W_k^+$ and $W_k^-$, in which blocks correspond to the  partition of $[0,1]$ into $k$ intervals of the same size and  values equal to the suprema/infima of $W$ over blocks 
$[\frac{i-1}{k},\frac{i}{k}) \times [\frac{i-1}{k},\frac{i}{k})$.  
This construction insures that 
\[
	W_k^{-} \leq W \leq W_k^{+}.
\]
Then, by definition \eqref{def:var-k},
\[
	\varphi_k(W_k^{-} ) \leq \varphi_k(W) \leq \varphi_k(W_k^{+} ).
\]
Also, there is a coupling $(\G, \G_k)$ such that
$\G \sim \M(W,n)$, $\G_k \sim \M(W_k^{-},n)$, and $\G_k \subseteq \G$. Then, we get that 
$ \chi(\G) \geq \chi(\G_k)$. 
If  we show that 
\begin{equation}\label{diff_Wpm}
\varphi_k(W_k^{+} ) \leq 	\varphi_k(W_k^{-} )  + O(k^{-1})
\end{equation}
then, 
applying Theorem \ref{T:block} to $W_k^-$ and $W_k^+$, we get the required estimate.  

It remains to show \eqref{diff_Wpm}. For the case when $W$ is Lipschitz, it is straightforward:   we observe  $|W_k^+-W_k^-| = O(k^{-1})$ and  apply Theorem \ref{T:norm}(c) with $S_\delta = \emptyset$. For the case when $W$ is increasing, we  consider graphon  $\hat{W}^-_k$   defined by 
\begin{align*}
	\hat{W}_k^-(x,y) &:= 	 
	\begin{cases}
			W_k^-(x+\frac1k, y +\frac1k), &\text{if $x,y \leq \frac{k-1}{k}$,}\\
			W_k^-(x+\frac1k-1, y +\frac1k-1), & \text{otherwise.}
		\end{cases}
\end{align*}
This graphon is equivalent to $W_k^-$ up to a preserving measure transformation. Therefore, $\varphi_k(\hat{W}_k^-) = \varphi_k(W_k^-)$.
 On the other hand, since $W$ is increasing,  we have
 \[
 \hat{W}_k^- \geq \hat{W}_k^+ := 	  \one_{[0,\frac{k-1}{k}) \times [0,\frac{k-1}{k})} W^+_k.
 \]
Therefore, $\varphi_k(\hat{W}_k^-)  \geq \varphi_k(\hat{W}_k^+) $. Finally, applying
Theorem \ref{T:norm}(c) with $S_\delta = [\frac{k-1}{k},1]$, we get that 
$\varphi_k(  \hat{W}_k^+) =  \varphi_k(W_k^+ ) + O(k^{-1})$. Combining the above, we get
\[
		\varphi_k(W_k^-) = \varphi_k(\hat{W}_k^-)  \geq \varphi_k(  \hat{W}_k^+) 
		 =  \varphi_k(W_k^+ ) + O(k^{-1}).
\]
This completes the proof of Theorem \ref{Thm:con-inc}.

\end{document}